\documentclass[12pt]{amsart}
\usepackage{amsmath,amsthm,amscd,amsfonts,amssymb,graphicx,color,mathabx,enumitem,mathtools,comment,amsrefs,tikz}
\usepackage[normalem]{ulem}

    \makeatletter
    \def\l@subsection{\@tocline{2}{0pt}{2.9pc}{5pc}{}}
    \def\l@subsubsection{\@tocline{2}{0pt}{5pc}{7.5pc}{}}
    \makeatother
\usepackage[margin=2.8cm]{geometry}

\usepackage{mathabx}
\usepackage[bookmarksnumbered, colorlinks, plainpages]{hyperref}
    \hypersetup{colorlinks=true, linkcolor=blue, citecolor=magenta, filecolor=magenta, urlcolor=cyan}
\allowdisplaybreaks
    \mathtoolsset{showonlyrefs,showmanualtags} 
\theoremstyle{plain} % definition 
\newtheorem{lemma}[equation]{Lemma} 
\newtheorem{proposition}[equation]{Proposition} 
\newtheorem{theorem}[equation]{Theorem} 
\newtheorem{corollary}[equation]{Corollary} 

\newtheorem{priorResults}{Theorem}

\theoremstyle{definition}
\newtheorem{definition}[equation]{Definition} 

\theoremstyle{remark}
\newtheorem{remark}[equation]{Remark}

\newtheorem*{ack}{Acknowledgment}

\numberwithin{equation}{section}

\usepackage{scalerel} 
\DeclareMathOperator*{\Expectation}{\scalerel*{\mathbb{E}}{\textstyle\sum}}

\usepackage{amsrefs} 
\begin{document}
\title[Cantor Szemer\'edi Theorem]{Szemer\'edi's Theorem Along Cantor Sets of Integers}

\author[Burgin]{Alex Burgin} 
    \address{School of Mathematics, Georgia Institute of Technology, Atlanta GA 30332, USA}
    \email{alexander.burgin@gatech.edu}
    \thanks{Research of AB supported in part by the Department of Education Graduate Assistance in Areas of National Need program at the Georgia Institute of Technology (Award \# P200A240169)}

\author[Fragkos]{Anastasios Fragkos}
\address{ School of Mathematics, Georgia Institute of Technology, Atlanta GA 30332, USA}
    \email {anastasiosfragkos@gatech.edu}

\author[Lacey]{Michael T. Lacey} 
    \address{ School of Mathematics, Georgia Institute of Technology, Atlanta GA 30332, USA}
    \email {lacey@math.gatech.edu}
    \thanks{Research of MTL and AB supported in part by grant  from the US National Science Foundation, DMS-2247254.}

 \author[Mena]{Dario Mena}
    \address{Centro de Investigaci\'{o}n en Matem\'{a}tica Pura y Aplicada \\ Escuela de Matem\'{a}tica, Universidad de Costa Rica}
    \email{dario.menaarias@ucr.ac.cr}
    \thanks{Research of DM partially supported by grant C1012 from Universidad de Costa Rica}

\author[Reguera]{Maria Carmen Reguera} \address{Universidad de M\'alaga}
    \email{m.reguera@uma.es}
    \thanks{Research of MCR supported by the Spanish Ministry of Science and Innovation through the projects RYC2020-030121-IAEI/10.13039/501100011033/ and PID2022-136619NB-I00 funded by MCIN/AEI/10.13039/501100011033/FEDER, UE.}
\begin{abstract}
    Let $\mathcal C= \{0=k_0<k_1<k_2 < \cdots\}$ be a Cantor set of integers, that is, the set of integers all of whose base $b$ digits lie in a fixed digit set $D$ of at least two digits, and suppose that $0\in D$.
We show that for every invertible measure preserving system $(X, \mathcal A , m ,T)$, every $A \in \mathcal A$ with $m(A)>0$, and every integer $\ell \geq 1$,
\begin{equation}
\liminf_{N-M \to \infty } \Expectation_{M\leq n<N}  m(A\cap T^{-k_n} A \cap \cdots
        \cap T^{-\ell k_n} A )>0.
\end{equation}
The averages above are uniform in the index $n$, and so the set of indices $n$ for which
$m(A\cap T^{-k_n} A \cap \cdots \cap T^{-\ell k_n} A )>0$ is syndetic.
By the Furstenberg correspondence principle, it follows that for every set of integers $A$ of positive upper Banach density
there is a syndetic set $B$ of indices such that, for each $n \in B$, the set of $a$ with
$a, a+k_n, \ldots, a+ \ell k_n \in A$ again has positive upper Banach density.

A Cantor set contains an IP-set, so the existence of such progressions, and the positivity of the density above for a single index,
already follow from the IP Szemer\'edi theorem of Furstenberg and Katznelson, and from the work of Kra and Shalom.
The point of this paper is instead the frequency with which the step sizes $k_n$ may be found,
measured in the natural enumeration of the Cantor set itself.
\end{abstract}

\maketitle

\tableofcontents 

\section{Introduction} 

Write $\mathbb N = \{1,2,\ldots\}$ and $\mathbb N_0 = \{0, 1,2,\ldots\}$. 
A subset $A$ of the integers is said to have \emph{positive upper Banach density} if 
\begin{equation}  \label{e:banach}
    d ^{\ast}(A) \coloneqq \limsup_{N-M\to\infty}  \frac{\lvert A\cap [M,N) \rvert} {N-M} >0. 
\end{equation} 
A set $B \subset \mathbb N_0$ is \emph{syndetic} if it has bounded gaps, that is, if there is an $L$ so that 
$B \cap [n, n+L] \neq \emptyset $ for every $n$.  A syndetic set has positive lower density, at least $(L+1) ^{-1}$; the converse is false.  

Recall that an \emph{IP-set} is built out of an infinite sequence of  
 integers $r_1  ,  r_2 ,\ldots $,  which need not be distinct: 
the IP-set is the collection of all the finite sums $ \sum_{n\in F} r_n $, 
where $F\subset \mathbb N$ is finite and non-empty.  
The theorem below is the IP Szemer\'edi theorem of Furstenberg and Katznelson 
\cite{FK85}*{Theorem A}, specialized to the powers of a single transformation.  

\begin{priorResults}  \label{t:FK}
    Let $(X, \mathcal A, m, T)$ be an invertible measure preserving system, let $A \in \mathcal A$ have $m(A)>0$, and let $\ell \geq 1$ be an integer. 
    Then, for every IP-set $\mathcal I$, there is an $r \in \mathcal I$ so that 
    \begin{equation}
        m\bigl(A \cap T ^{-r} A \cap \cdots \cap T ^{-\ell r} A \bigr) >0. 
    \end{equation}
\end{priorResults}

By the Furstenberg correspondence principle, Theorem \ref{t:FK} is already a \emph{density} statement about the integers.  
Namely, if $A \subset \mathbb Z$ has positive upper Banach density, and $\mathcal I$ is an IP-set, 
then there is an $r \in \mathcal I$ for which the set of $a$ with 
$a, a+r , \ldots, a+ \ell r \in A$ again has positive upper Banach density.  
Indeed, it follows that the set of \emph{good differences}
\begin{equation} \label{e:good}
\mathsf{G} = \mathsf{G}(A , \ell ) \coloneqq 
\bigl\{ r \in \mathbb N \colon d ^{\ast} \bigl( A \cap (A-r) \cap \cdots \cap (A - \ell r) \bigr) >0 \bigr\} 
\end{equation}
meets every IP-set, that is, $\mathsf G$ is an IP$^{\ast}$ set.  
Such a set is syndetic: the complement of a set which is not syndetic contains arbitrarily long intervals, and a set which contains arbitrarily long intervals contains an IP-set.  

\smallskip 

The sets of integers of interest to us are these.  
Let $\mathcal{C}_{b,D}$ denote a \emph{classical integer Cantor set}, defined by a base $b\geq 3$ 
and a collection of digits 
$ D \subset \{0,1 , \ldots, b-1\}$ with $ 2\leq \lvert D \rvert < b$.  
That is, 
\begin{equation} \label{e:Cantor}
\mathcal{C}_{b, D}  \coloneqq 
\Bigl\{ 
 \sum_{j=0}^k d_j b ^{j} 
 \colon  k\in \mathbb{N}_0 ,\ d_j \in  D
\Bigr\}. 
\end{equation}
The most classical version is the middle thirds Cantor set $\mathcal{C}_{3,\{ 0, 2 \}}$, that is base $b=3$ 
and allowed digits $D= \{0,2\}$.  
Throughout the paper we assume that $0 \in D$, and we write 
$\mathcal C = \mathcal C_{b,D} = \{ 0=k_0 < k_1 < k_2 < \cdots\}$ for the increasing enumeration of the Cantor set.  

A Cantor set with $0 \in D$ contains an IP-set.  Indeed, fixing a non-zero $d \in D$, 
the finite sums of the integers $d b ^{j}$, $j \geq 0$, all have their base $b$ digits in $\{0,d\} \subset D$.  
Therefore, Theorem \ref{t:FK} applies to the IP-set inside $\mathcal C$, and so already provides 
elements of $\mathcal C$ which are good differences.  
The same remark applies to the recent work of Kra and Shalom \cite{kra2025ergodicaverageslargeintersection}, 
who show that for IP-sets with rational spectra---a class which includes the IP-set just described---the  
IP averages above have the same characteristic factors as the full linear averages, and who obtain 
the optimal lower bounds for progressions of length three and four.  
It should be emphasized that the implication runs from IP-sets to Cantor sets, and not the other way:  
the Cantor set is the \emph{larger} set, so that locating a good difference inside the IP-set that it contains is the stronger statement.  

\smallskip 

What these results do not address is the \emph{frequency} with which the good differences occur inside the Cantor set.  
A Cantor set comes with its own natural enumeration $\{k_n \colon n \geq 0\}$, and the question we take up is how often 
$k_n$ is a good difference, as measured in the index $n$.  

It should be said at once that the two properties of $\mathsf G$ recorded above do not answer this question.  
Syndeticity of $\mathsf G$ carries no information about $\mathsf G \cap \mathcal C$, since $\mathcal C$ has density zero: 
a syndetic set of integers can be disjoint from a Cantor set, as one sees by taking $s = \gcd (D)>1$, so that $\mathcal C \subset s \mathbb Z$, 
and taking the syndetic set of integers not divisible by $s$.  
Nor does the IP$^{\ast}$ property settle it.  That property places an element of $\mathsf G$ inside every \emph{infinite} IP-set.  
What is wanted here is, for some fixed $i$, an element of $\mathsf G$ in each of the blocks 
$b ^{i}k_u + \bigl( \mathcal C \cap [0,b ^{i}) \bigr)$, $u \geq 0$, into which Lemma \ref{l:Diff}(ii) decomposes $\mathcal C$.  
These blocks are finite, and they are translated.  

\smallskip 

Our main theorem is the following.  It is stated in terms of averages over intervals of indices 
$M \leq n < N$, with $N - M \to \infty$; the uniformity in $M$ is exactly what delivers the frequency statement.  

\begin{theorem}  \label{t:CSZ}
    Let $\mathcal C = \mathcal C_{b,D}=\{0=k_0<k_1 < \cdots\}$ be a Cantor set with  $0\in D$.  
    Then, for all integers $\ell \geq 1 $, all  invertible  measure preserving systems $(X,\mathcal A, m, T)$, 
    and  all non-negative  $f \in L ^{\infty}(X)$ which are not identically zero, we have 
    \begin{equation} \label{e:CSZ}
        \liminf_{N-M\to\infty} \Expectation_{M \leq n < N} 
        \int_X \prod_{j=0}^{\ell}  T ^{j k_n} f  \;dm >0. 
    \end{equation}
    \end{theorem}

\begin{corollary} \label{c:syndetic}
    With the notation and assumptions of Theorem \ref{t:CSZ}, the set of indices 
    \begin{equation}
    \Bigl\{ n \in \mathbb N_0 \colon \int_X \prod_{j=0}^{\ell}  T ^{j k_n} f  \;dm >0 \Bigr\} 
    \end{equation}
    is syndetic.  
\end{corollary}

\begin{proof}
    Write $c>0$ for the left hand side of \eqref{e:CSZ}, and suppose that the set of indices above is not syndetic.  
    Then there are intervals $[M_i, N_i)$ with $N_i - M_i \to \infty$ which contain no such index.  
    The average in \eqref{e:CSZ} over $[M_i,N_i)$ is then zero for every $i$, contradicting \eqref{e:CSZ}.  
\end{proof}

Passing to the integers, we obtain the frequency statement that is the point of this paper.  

\begin{theorem}  \label{t:CSZ-Integers}
    Let $\mathcal C =\{0=k_0<k_1 < \cdots\}$ be a Cantor set with $0\in D$, and let $\ell \geq 1$ be an integer.  
    For every $A \subset \mathbb Z$ of positive upper Banach density, the set of indices 
    \begin{equation} \label{e:B}
        B \coloneqq \{ n \in \mathbb N \colon k_n \in \mathsf G (A, \ell ) \} 
        = \bigl\{ n \in \mathbb N \colon d ^{\ast} \bigl( A \cap (A-k_n) \cap \cdots \cap (A - \ell k_n) \bigr) >0 \bigr\} 
    \end{equation}
    is syndetic.  In particular, for every $n \in B$, the set $A$ contains a non-degenerate $(\ell +1)$-term arithmetic progression with step size $k_n$.  
    \end{theorem}

\begin{proof}
By the Furstenberg correspondence principle \cite{MR0498471}*{Thm 1.1}, there is an invertible measure preserving system 
$(X,\mathcal A, m ,T)$ and a set $\tilde A \in \mathcal A$ with $m (\tilde A) = d ^{\ast}(A)>0$, such that for every finite collection of integers $r_1 , \ldots, r_j$, 
\begin{equation} \label{e:correspondence}
d ^{\ast} \bigl( A \cap (A - r_1) \cap \cdots \cap (A-r_j) \bigr) 
\geq 
m \bigl( \tilde A \cap T ^{-r_1} \tilde A \cap \cdots \cap T ^{-r_j} \tilde A \bigr) . 
\end{equation}
We apply \eqref{e:correspondence}  for each fixed index $n$, with $(r_1 , \ldots, r_ \ell ) = (k_n , 2k_n , \ldots, \ell k_n)$.  
It follows that $B$ contains the set of indices $n \geq 1$ for which 
$m ( \tilde A \cap T ^{-k_n} \tilde A \cap \cdots \cap T ^{- \ell k_n} \tilde A )>0$, and the latter is syndetic by 
Theorem \ref{t:CSZ} and Corollary \ref{c:syndetic}, applied to $f = \mathbf 1_{\tilde A}$.  
(Discarding the index $n=0$, for which $k_0=0$ and the statement is trivial, does not affect syndeticity.)  
\end{proof}

\begin{remark}  \label{r:order}
    The correspondence principle is applied above for one index $n$ at a time, and this is essential.  
    The averages over the index $n$ are taken \emph{after} the passage from the integers to the measure preserving system, 
    and it is only the syndeticity of the set of good indices which survives that passage.  
    An average over $x$ and $n$ taken simultaneously, as in 
    $\liminf_{N} \Expectation_{x \in [N]} \Expectation_{n \in [N]} \prod_{j=0}^{\ell}\mathbf 1_{A} (x - jk_n)$, 
    is not a formal consequence of Theorem \ref{t:CSZ}: 
    the limit in $x$ is taken first in the correspondence principle, and in general 
    $\lim_{n} \lim _{m} a_{m,n} = a$ does not imply $\lim_n a_{n,n}=a$.  
\end{remark}

\begin{remark} \label{r:novelty}
    It is useful to separate the several questions that arise here.  
    That a set of positive density contains $(\ell+1)$-term progressions is Szemer\'edi's theorem.  
    That the step size may be required to lie in a prescribed IP-set, and that a positive density of such progressions is then obtained, 
    is the theorem of Furstenberg and Katznelson, Theorem \ref{t:FK}; by the remarks above it covers Cantor sets as well.  
    The optimal lower bounds for the density of progressions, with step size restricted to a rational IP-set, are due to Kra and Shalom, 
    for progressions of length three and four.  
    The question addressed here is instead one of frequency: among the step sizes $k_n$ that a Cantor set makes available, 
    how many are good differences?  Theorem \ref{t:CSZ-Integers} answers that a syndetic set of them are.  
    Frequency questions of this kind seem to have received relatively little attention.  
\end{remark}

Our proof of Theorem \ref{t:CSZ} follows the maximal-factor strategy of the ergodic proof of Szemer\'{e}di's Theorem \cites{MR0498471,MR670131}, 
carried out so that all averages are uniform in the index, as in \eqref{e:CSZ}.  
Section \ref{s:prelim} sets this up.  The CSZ property \eqref{e:CSZdef} of a measure preserving system carries a dilation parameter, which the compact case requires, 
and Theorem \ref{t:CSZ} follows from three statements: the property passes to weakly mixing and to compact extensions of a CSZ factor (Lemma \ref{l:Extensions}); 
it is preserved by suprema of chains of factors (Lemma \ref{l:chain}); and every proper factor has an extension of one of these two kinds (Lemma \ref{l:dichotomy}).  
These give the result for ergodic systems on compact metric spaces, and Proposition \ref{p:reduction} reduces the general case to that one.  
The arithmetic of the Cantor set enters through two facts: the block structure of its enumeration (Lemma \ref{l:Diff}), 
and a uniform estimate for exponential sums $\sum e(\beta k_n)$ over intervals of indices (Lemma \ref{l:irrational}).  

The Cantor set imposes a real constraint at two steps.  For weakly mixing extensions, in \S\ref{sec:wm}, the van der Corput argument leads to correlations 
along the differences $k_{n+h}-k_n$; these take few values, so the correlation for an individual shift $h$ is not small.  
We bound each shift by a non-negative quantity, and only then average over the shifts, using the spectral measure of $f\otimes \overline f$ on the 
relatively independent self-joining.  For compact extensions, in \S\ref{sec:compact}, the argument of Furstenberg that passes through van der Waerden's theorem 
produces return times which are dilates, and a Cantor set is not closed under dilation.  We follow instead the argument of Furstenberg, Katznelson and Ornstein \cite{MR670131}, 
whose return times are translates, which a Cantor set does absorb.  A Cantor set version of van der Waerden's theorem is proved in \S\ref{sec:vdW}, independently of the rest.  

Further questions that this result presents to us include: 
\begin{itemize}
    \item The question of quantitative bounds for these sorts of configurations is certainly challenging, and could be of some interest. 
  In the special case of $\ell=1$,  examples of power savings bounds are contained in \cite{firstpaper}. 
  
    \item  Again, for $\ell=1$, there is a polynomial version of this result \cite{firstpaper}. The result above should surely hold for polynomial configurations as well. 
    
    \item  The higher dimensional analog could have different formulations as well.  
    
\end{itemize}

\begin{ack}
    We thank Or Shalom for a careful reading of an earlier version of this paper, and for the comments which led to the
    present formulation of Theorems \ref{t:CSZ} and \ref{t:CSZ-Integers}.  
\end{ack}

\emph{AI Usage:} The initial version of this paper was written by the authors. In the second revision, the authors used Claude AI, at different levels for a check of arguments, and proof editing.  After that, Claude AI was used to organize a \href{https://github.com/lacey-gatech/CantorSzemeredi/blob/main/README.md}{partial formalization} of the paper in LEAN.  The statement of the result formalized can be read at this link. Two results are assumed for the formalization, first the Dichotomy Lemma \ref{l:dichotomy}, and second, the standard ergodic decomposition cited in the proof of Proposition \ref{p:reduction}. 
\section{Preliminaries} \label{s:prelim} 

\subsection{Notation} 
We use standard Vinogradov notation, thus $A\ll B$ means that $\lvert A\rvert \leq C \lvert B\rvert$, 
for an absolute constant $C$. By $A=B+O(\delta)$ we mean that $A-B \ll \delta$.  

For integers $N$, we set $[N] = \{1,2, \ldots, N\}$. 
For a finite set $F$ and $f \colon F\to \mathbb C$, we write 
\begin{equation}
    \Expectation_F f = \frac{1}{\lvert F\rvert} 
    \sum_{n\in F} f(n). 
\end{equation}
And we will frequently write $\Expectation_{[N]}=\Expectation_N$.  
Averages over an interval of indices are written $\Expectation_{M\leq n<N}$, 
and a limit $\lim_{N-M\to\infty}$ is a limit over all such intervals, as their length tends to infinity.  
Such averages, and the limits of them, are referred to as \emph{uniform}.  
We will use repeatedly the following elementary observation, in the form given in Corollary \ref{c:syndetic}: 
if a non-negative sequence $(a_n)$ satisfies $\liminf_{N-M\to\infty} \Expectation_{M\leq n<N} a_n>0$, 
then $\{n \colon a_n >0\}$ is syndetic.  

\subsection{Measure Preserving Systems} \label{sub:mps}

We recall standard notions around measure preserving systems.  
In particular, $(X,\mathcal A,m,T)$ 
is a \emph{measure preserving system} if $(X,\mathcal A,m)$ is a probability space, and $T \colon X\to X$ preserves $m$, in that $m(T^{-1} A)=m(A)$ for all measurable $A\in \mathcal A$.  
By $\Expectation_X f $ we mean $\int_X f \;dm$. 

\emph{Standing assumption.}  From here to the end of \S\ref{sec:compact}, unless stated otherwise, $(X,\mathcal A,m,T)$ is an invertible measure preserving system 
in which $X$ is a compact metric space and $\mathcal A$ is its Borel $\sigma$-algebra; in particular $m$ is a Radon measure.  
A system is \emph{ergodic} if its $T$-invariant sets have measure $0$ or $1$; this is assumed only where stated.  
Theorem \ref{t:CSZ} concerns arbitrary invertible systems, and Proposition \ref{p:reduction} in \S\ref{sub:reduction} shows that it suffices to treat 
ergodic systems satisfying the standing assumption.

In different settings, the underlying space $(X,\mathcal A, m )$ remains the same, while we change the measure preserving operator from $T$, to some power $T^k$. 
In this case, we  just indicate the change in the operator.  
In other cases, we have  $T$ invariant sigma algebras, $\mathcal D \subset \mathcal E \subset \mathcal A$. 
In that case, we may refer to the system by its sigma algebra. 

The system is called \emph{weak mixing} if for all $A, B\in \mathcal A$, one has 
\begin{equation} \label{e:weaklyMixingDef}
    \lim_N \Expectation_N \lvert 
    m(A\cap T^{-n}B) - m(A)m(B) \rvert =0. 
\end{equation}
This condition for sets extends to functions. For $f\in L^2(X)$ with mean zero,  
\begin{equation} \label{wm-functions}
    \lim_N \Expectation_N  
    \lvert \langle f, T^n g \rangle \rvert =0, \qquad g \in L^2(X). 
\end{equation}
Equivalently, the spectral measure $\sigma _f$ of a mean zero $f \in L^2(X)$ has no atoms.  
Here, and below, for a unitary operator $U$ on a Hilbert space and a vector $v$, the \emph{spectral measure} $\sigma _v$ 
is the finite positive Borel measure on $\mathbb T = \mathbb R/\mathbb Z$ with 
\begin{equation} \label{e:spectral}
\langle U ^{n}v, v\rangle = \int_{\mathbb T} e(n \alpha ) \;d \sigma _v ( \alpha ), \qquad n \in \mathbb Z . 
\end{equation}
When no operator is named, $U$ is the Koopman operator $f \mapsto f\circ T$ of the system at hand.

\bigskip 

A function $f\in L^2(X)$ is called \emph{compact} if the set $\overline{ \{T^n f \colon n\in \mathbb{Z}\} }$ is a compact subset of $L^2 (X)$.  The system is \emph{compact} if every $f\in L^2(X)$  is compact.  In this case $T$ has pure point spectrum: $L^2(X)$ is spanned by eigenfunctions of $T$.  
The eigenvalues need not be roots of unity; an irrational rotation of the circle is a compact system.

\smallskip 
Let $\mathcal D \subset \mathcal A$ be a $T$ invariant 
sigma algebra. The system 
$(X,\mathcal D, m , T)$ is said to be a \emph{factor} of 
the original system $(X,\mathcal A,m ,T)$.  
(Here, there is a slight abuse of notation, in that we use the same notation for the set function $m$ restricted to the collection $\mathcal D$.) 
Factors are then naturally identified with their sigma algebras. 
And, given $\mathcal D\subset \mathcal E \subset \mathcal A$, 
we refer to $ \mathcal E$ as an extension of $\mathcal D$.  

With $\mathcal D\subset \mathcal E \subset \mathcal A$, we define a conditional inner product 
\begin{equation}
    \langle f,g \rangle _{\mathcal E\mid \mathcal D } 
    = \mathbb{E}_m  (f\overline g \mid \mathcal D) 
    \in L^2(\mathcal D) , \qquad 
    f,g\in L^2(\mathcal E). 
\end{equation}
This conditional inner product defines conditional  $L^2$ norms by 
$ \lVert f \rVert _{ \mathcal E \mid \mathcal  D }$. 
A conditional Cauchy--Schwarz and triangle inequality hold for this norm. 

\begin{definition}
    The Hilbert module $L^2(\mathcal E\mid \mathcal D)$ 
    are all those $f\in L^2(\mathcal E)$ for which the conditional norm  $\lVert f \rVert _{ \mathcal E \mid \mathcal  D } \in L^\infty (\mathcal D)$. 
\end{definition}

\begin{definition}
    A function $f\in L^2 (\mathcal E) $ is said to be \emph{conditionally weakly mixing} if  
    \begin{equation}
        \lim_N \Expectation_{N} 
        \lvert \langle T^nf,f\rangle _{\mathcal E\mid \mathcal D} 
        \rvert =0, \quad \textup{in $L^2(\mathcal D)$} . 
    \end{equation}
The extension $\mathcal E$ of $\mathcal D$ is said to be \emph{weakly mixing} if every $f\in L^2 (\mathcal E \mid \mathcal D) $ with $\mathbb E(f\mid \mathcal D)=0$  is conditionally weakly mixing. 
\end{definition}

We will use the definition of a weakly mixing extension only through the following consequence, 
which is phrased in terms of the relatively independent self-joining.  
As $X$ is a compact metric space, $m$ has a disintegration $m = \int m_x \;dm(x)$ over $\mathcal D$ \cite{MoreiraNotes}*{Theorem 7.2.1}, 
so that $\mathbb E(a\mid \mathcal D)(x) = \int a \;dm_x$ for $a \in L ^{\infty}(X)$ and almost every $x$.  
The \emph{relatively independent self-joining} of $m$ over $\mathcal D$ is the probability measure 
$\lambda \coloneqq \int m_x \otimes m_x \;dm(x)$ on $X \times X$; it appears, for instance, in the proof of \cite{MoreiraNotes}*{Proposition 8.6.1}.  
It has the two properties we need.  First, for $a, a' \in L ^{\infty}(X)$, 
\begin{equation} \label{e:joining}
\int_{X \times X} a \otimes a' \;d \lambda = \int_X \mathbb E(a\mid \mathcal D)\, \mathbb E(a'\mid \mathcal D)\;dm . 
\end{equation}
Second, $\lambda$ is invariant under $R \coloneqq T \times T$.  Indeed, as $\mathcal D$ is $T$-invariant, $\mathbb E(a\circ T\mid \mathcal D) = \mathbb E(a\mid \mathcal D)\circ T$ 
\cite{MoreiraNotes}*{Lemma 8.1.4}; so by \eqref{e:joining}, and as $T$ preserves $m$, 
$\lambda(R ^{-1}(B\times B')) = \int (\mathbb E(\mathbf 1_B\mid \mathcal D)\, \mathbb E(\mathbf 1_{B'}\mid \mathcal D))\circ T\;dm = \lambda(B\times B')$ 
for Borel sets $B, B'$.  Rectangles generate the Borel sets of $X \times X$ and are closed under intersection, so $\lambda \circ R ^{-1} = \lambda$.  

\begin{lemma} \label{l:atomless}
Let $f \in L ^{\infty}(\mathcal E)$ be conditionally weakly mixing, and set $G \coloneqq f \otimes \overline f$.  Then 
\begin{enumerate}
\item[(i)] the spectral measure $\sigma _f$ of $f$, for the Koopman operator of $T$ on $L^2(X)$, has no atoms; 
\item[(ii)] the spectral measure $\sigma _G$ of $G$, for the Koopman operator of $R$ on $L^2(\lambda)$, has no atoms, and 
\begin{equation} \label{e:psiIdentity}
\langle R ^{n}G, G\rangle _{L^2(\lambda)} = \bigl\lVert \langle T ^{n}f, f\rangle _{\mathcal E\mid \mathcal D}\bigr\rVert _{L^2(\mathcal D)} ^2 , 
\qquad n \in \mathbb Z . 
\end{equation}
\end{enumerate}
\end{lemma}

\begin{proof}
The identity \eqref{e:psiIdentity} is \eqref{e:joining} with $a = T ^{n}f\cdot \overline f$ and $a' = \overline{T ^{n}f}\cdot f = \overline a$, 
since $R ^{n}G\cdot \overline G = a \otimes a'$; the right side of \eqref{e:joining} is then $\int \lvert \mathbb E(T ^{n}f\cdot\overline f\mid \mathcal D)\rvert ^2 \;dm$.  

We use Wiener's lemma: for a finite positive measure $\sigma$ on $\mathbb T$, with $\widehat\sigma (n) = \int e(n \alpha )\;d \sigma (\alpha )$, 
\begin{equation}
\sum _{\theta} \sigma ( \{\theta\}) ^2 = \lim_N \Expectation_{1\leq n\leq N} \lvert \widehat\sigma (n)\rvert ^2 . 
\end{equation}

Both conclusions are unaffected by scaling $f$, so we may assume $\lVert f\rVert _{\infty} \leq 1$.  
Let $a_n \coloneqq \lVert \langle T ^{n}f,f\rangle _{\mathcal E\mid \mathcal D}\rVert _{L ^{1}(\mathcal D)}$.  
The functions $\lvert \langle T ^{n}f,f\rangle _{\mathcal E\mid \mathcal D}\rvert$ are non-negative and bounded by $1$, so 
\begin{equation}
\Expectation_{1\leq n\leq N} a_n = \Bigl\lVert \Expectation_{1\leq n\leq N} \bigl\lvert \langle T ^{n}f,f\rangle _{\mathcal E\mid \mathcal D}\bigr\rvert \Bigr\rVert _{L ^{1}(\mathcal D)} 
\leq \Bigl\lVert \Expectation_{1\leq n\leq N} \bigl\lvert \langle T ^{n}f,f\rangle _{\mathcal E\mid \mathcal D}\bigr\rvert \Bigr\rVert _{L ^{2}(\mathcal D)} \to 0 , 
\end{equation}
the limit being the definition of conditional weak mixing.  
Now $\widehat{\sigma _f}(n) = \langle T ^{n}f,f\rangle = \int \langle T ^{n}f,f\rangle _{\mathcal E\mid \mathcal D}\;dm$, so $\lvert \widehat{\sigma_f}(n)\rvert \leq a_n$; 
and $0 \leq \widehat{\sigma _G}(n) \leq a_n$ by \eqref{e:psiIdentity}, since the square of a function bounded by $1$ is at most its modulus.  
As $a_n \leq 1$, it follows that 
$\Expectation_{1\leq n\leq N} \lvert \widehat\sigma (n)\rvert ^2 \leq \Expectation_{1\leq n\leq N} a_n \to 0$ for $\sigma \in \{\sigma _f, \sigma _G\}$, 
and Wiener's lemma shows that neither measure has an atom.  
\end{proof}

\begin{definition}
    A function $f\in L^2 (\mathcal E \mid \mathcal D) $ is said to be \emph{conditionally compact} if for every $\epsilon>0$ there 
    are a finite number of functions $g_1 ,\ldots,g_r \in L^2(\mathcal E\mid \mathcal D)$ so that 
    \begin{equation}
        \sup_{n\in \mathbb Z} 
        \min_{1\leq q\leq r} 
        \bigl\lVert 
        \lVert T^n f - g_q \rVert_{\mathcal E\mid  \mathcal D} 
        \bigr\rVert_{L^\infty (\mathcal { D})} < \epsilon. 
    \end{equation}

    The extension $\mathcal E$ of $\mathcal D$ is said to be \emph{compact} if conditionally compact functions  are dense in  $L^2 (\mathcal E) $.  
\end{definition}

\subsection{Cantor Sets} 

We defined the Cantor sets of interest in \eqref{e:Cantor}. Here, we collect the properties of them needed to prove the main theorem.  
Throughout, $\mathcal C = \mathcal C_{b,D} = \{0=k_0<k_1 < \cdots\}$, with $0 \in D$, and we abbreviate 
\begin{equation} \label{e:delta}
\delta \coloneqq \lvert D\rvert \geq 2 . 
\end{equation}

The following identities describe the enumeration of $\mathcal C$; they are from \cite{firstpaper}*{Prop.~2.2}, 
and as everything below rests on them we include the short proof.  

\begin{lemma} \label{l:Diff}
    Let $\mathcal C=(k_n)_{n\geq 0}$ be an integer Cantor set in base $b$, with digit set $D=\{0=d_0<\cdots <d_{\delta -1}\}\subset \{0,\ldots ,b-1\}$. Then: 
    \begin{enumerate}
        \item[(i)] For any $n\in \mathbb{N}_0$, written in base $\delta $ as $n=\sum_{c=0}^tn_c\delta ^c$ with $n_c\in \{0,\ldots ,\delta -1\}$, we have $k_n=\sum_{c=0}^td_{n_c}b^c$.  
        \item[(ii)] For any $i\geq 0$, $n\geq 0$, and $0\leq j<\delta ^i$, we have $k_{\delta ^in+j}=b^ik_n+k_j$. 
    \end{enumerate}
\end{lemma}

\begin{proof}
For (i), define $\varphi \colon \mathbb N_0 \to \mathbb N_0$ by $\varphi(n) \coloneqq \sum_c d_{n_c} b ^{c}$, where the $n_c$ are the base $\delta$ digits of $n$.  
Since $d_0 = 0$, appending leading zero digits to $n$ does not change $\varphi(n)$, and the base $b$ digits of $\varphi (n)$ are exactly the $d_{n_c}$, all of which lie in $\{0 , \ldots, b-1\}$.  
So $\varphi(n) \in \mathcal C$; conversely, every element of $\mathcal C$ is $\varphi(n)$ for the $n$ whose base $\delta$ digits index its base $b$ digits in $D$, 
padding with the digit $0 = d_0$ where needed.  Thus $\varphi$ maps $\mathbb N_0$ onto $\mathcal C$.  It is strictly increasing, and so injective: if $n<n'$, let $c$ be the highest position at which their base $\delta$ digits differ, 
so that $n_c < n'_c$.  Then $\varphi(n)$ and $\varphi(n')$ have the same base $b$ digits above position $c$, and $d_{n_c} < d_{n'_c}$ at position $c$, so $\varphi(n)<\varphi(n')$.  
A strictly increasing map from $\mathbb N_0$ onto $\mathcal C$ is the increasing enumeration, so $k_n = \varphi(n)$.  

For (ii), write $j = \sum_{c<i} j_c \delta ^{c}$.  As $j< \delta ^{i}$, the base $\delta$ digits of $\delta ^{i}n + j$ are $j_0 , \ldots, j_{i-1}$ in the positions below $i$, 
followed by the digits of $n$, shifted up by $i$.  So by (i), $k_{\delta ^{i}n+j} = \sum_{c<i} d_{j_c}b ^{c} + \sum_c d_{n_c}b ^{c+i} = k_j + b ^{i}k_n$.  
\end{proof}

Note that part (ii) says that the blocks of indices 
$[\delta ^i n , \delta ^{i} (n+1))$, which partition $\mathbb N_0$ into consecutive intervals of length $\delta ^{i}$, 
correspond to the blocks $b ^{i} k_n + \bigl( \mathcal C \cap [0,b ^{i}) \bigr)$ of the Cantor set.  
This is the property that permits us to prove syndeticity, in place of mere positive density.  
It also controls the differences $k _{n+h}-k_n$, which are the subject of the van der Corput argument of \S\ref{sec:wm}.  

\begin{lemma} \label{l:differences}
Fix integers $h \geq 1$ and $\ell \geq 1$ with $\delta ^{\ell} > h$.  
For $0 \leq c < \delta ^{\ell}-h$ set $P_c \coloneqq \{ n \geq 0 \colon n \equiv c \mod \delta ^{\ell}\}$, 
and $\mathbb P \coloneqq \bigcup _{0\leq c < \delta ^{\ell}-h} P_c$. Then 
\begin{enumerate}
\item[(i)] For every $n \in P_c$, the difference $k _{n+h}-k_n$ is equal to the constant $\Delta _c \coloneqq k _{c+h}-k_c$.  
\item[(ii)] $\bigl\lvert \mathbb P \cap [M,N) \bigr\rvert \geq (1- h \delta ^{- \ell})(N-M) - \delta ^{\ell}$, for every interval $[M,N)$.  
\end{enumerate}
In particular, on a set of indices of uniform density at least $1-h \delta ^{-\ell}$, the difference $k_{n+h}-k_n$ takes at most $\delta ^{\ell}$ values, 
and each of the sets on which it is constant is a finite union of arithmetic progressions of step $\delta ^{\ell}$.  
\end{lemma}

\begin{proof}
For (i), write $n = \delta ^{\ell}a +c$ with $0 \leq c < \delta ^{\ell}-h$.  Then also $n+h = \delta ^{\ell}a + (c+h)$ with $0\leq c+h< \delta ^{\ell}$, 
so two applications of Lemma \ref{l:Diff}(ii) give $k_n = b ^{\ell}k_a + k_c$ and $k _{n+h} = b ^{\ell}k_a + k _{c+h}$.  Subtract.  
For (ii), $\mathbb P$ is the union of $\delta ^{\ell}-h$ residue classes modulo $\delta ^{\ell}$.  
\end{proof}

The distribution of Cantor sets in residue classes has been investigated in several papers, including \cites{ERDOS199899, SAAVEDRA-ARAYA_2026}.  
In \cite{firstpaper} it is shown that, for each $q \in \mathbb N$, the proportion of indices $n < N$ with $k_n \equiv a \mod q$ converges as $N \to \infty$, 
and that the limit is positive for $a = 0$.  The limiting distribution is in general not uniform on $\mathbb Z_q$, 
and the proportions do \emph{not} converge uniformly, in the sense of $\lim_{N-M\to\infty}$.  
Nevertheless, the set of indices $n$ with $q \mid k_n$ is not merely of positive density, but is syndetic.  
This is used only in Theorem \ref{t:compactSystem}, which is not needed for the main result; 
we give the short proof, as it is a good illustration of the block structure of Lemma \ref{l:Diff}(ii).  

\begin{lemma} \label{l:syndeticResidues}
Let $\mathcal C = \mathcal C_{b,D} = \{0=k_0<k_1 < \cdots\}$ be a Cantor set with $0 \in D$.  
For every integer $q \geq 1$, the set 
\begin{equation}
S_q \coloneqq \{ n \in \mathbb N_0 \colon k_n \equiv 0 \mod q \} 
\end{equation}
is syndetic.  
\end{lemma}

\begin{proof}
For $i \geq 1$ let $R_i \subset \mathbb Z_q$ be the set of residues $k_j \bmod q$ for $0 \leq j < \delta ^{i}$; 
by Lemma \ref{l:Diff}(i), 
\begin{equation}
R_i = \Bigl\{ \sum_{c=0}^{i-1} d_c b ^{c} \bmod q \colon d_0 , \ldots, d_{i-1} \in D \Bigr\} . 
\end{equation}
Since $0 \in D$, the sets $R_i$ increase with $i$, hence  they stabilize: there is an $i_0$ with $R_i = R \coloneqq \bigcup_{i} R_i$ for all $i \geq i_0$.  

Let $c_0 \geq 0$ and $\tau \geq 1$ be such that $b ^{c+\tau} \equiv b ^{c} \bmod q$ for all $c \geq c_0$, and set 
$Y= \{ d b ^{c} \bmod q \colon d \in D ,\ c \geq c_0\}$, a finite subset of $\mathbb Z_q$.  
Each element of $Y$ is of the form $db ^{c} \bmod q$ for infinitely many exponents $c$.  
Therefore any finite sum of elements of $Y$, repetitions allowed, can be realized as 
$\sum_{c} d_c b ^{c} \bmod q$ with the exponents $c$ \emph{distinct}, and so lies in $R$.  
The collection of all such sums is a non-empty finite sub-semigroup of the group $\mathbb Z_q$, and hence is the subgroup $H \subset \mathbb Z_q$ generated by $Y$.  
Thus $H \subset R$.  As $H$ is finite, there is an $i_1$ so that every element of $H$ is such a sum with all exponents less than $i_1$, whence $H \subset R_i$ for $i \geq i_1$.  

Now fix $i \geq \max (c_0 , i_0 , i_1)$ and let $n \in \mathbb N_0$ be arbitrary.  
Writing $k_n = \sum_{c} d_c b ^{c}$ with $d_c \in D$, we have $b ^{i} k_n = \sum_{c} d_c b ^{c+i}$, and all the exponents $c+i$ are at least $c_0$.  
Hence $b ^{i} k_n \bmod q \in H$, and therefore also $- b ^{i} k_n \bmod q \in H \subset R_i$.  
Select $0 \leq j < \delta ^{i}$ with $k_j \equiv - b ^{i} k_n \bmod q$.  By Lemma \ref{l:Diff}(ii), 
\begin{equation}
k _{\delta ^{i} n+j} = b ^{i} k_n + k_j \equiv 0 \mod q . 
\end{equation}
That is, $S_q$ meets the block of indices $[\delta ^{i} n , \delta ^{i}(n+1))$, for every $n \geq 0$.  
As these blocks partition $\mathbb N_0$ into consecutive intervals of length $\delta ^{i}$, the set $S_q$ is syndetic, with gaps at most $2 \delta ^{i}$.  
\end{proof}

The input from our first paper is the following uniform exponential sum estimate.  

\begin{lemma} \label{l:irrational}
Let $\mathcal C = \mathcal C_{b,D} = \{0=k_0<k_1 < \cdots\}$ be a Cantor set with $0\in D$, and let 
$E_0 \subset \mathbb T$ be the set of $\beta$ for which $\gcd(D)\, \beta$ is a $b$-adic rational modulo $1$; it is a countable set of rationals.  
Then for every $\beta \notin E_0$, and in particular for every irrational $\beta$, 
\begin{equation} \label{e:irrationalUniform}
\lim_{L\to\infty} \sup_{M \geq 0} \Bigl\lvert \Expectation_{M\leq n<M+L} e ( \beta k_n) \Bigr\rvert =0 . 
\end{equation}
\end{lemma}

This is the case of the uniform estimate in \cite{firstpaper}*{\S 4} in which the frequency attached there to the sum of digits is $0$; 
in the notation of that paper, $( \beta ,0)$ is exceptional exactly when $\gcd(D)\beta$ is a $b$-adic rational.  
The supremum over the starting index $M$ is the essential point, and the reason for the word `uniform' throughout this paper.  
As all of \S\ref{sec:wm} rests on it, we include the proof, which in this case is short.  

\begin{proof}
Write $s \coloneqq \gcd(D)$, and recall $\delta = \lvert D\rvert$.  

\emph{Reduction to $s=1$.}  The digit set $D' \coloneqq \{d/s \colon d \in D\}$ again satisfies $0 \in D' \subset \{0 , \ldots, b-1\}$ and $2 \leq \lvert D'\rvert < b$, and $\gcd(D')=1$.  
The elements of $\mathcal C$ are $k_n = s k'_n$, where $k'_0 < k'_1 < \cdots$ are the elements of $\mathcal C_{b,D'}$, so that $e(\beta k_n) = e(s\beta\, k'_n)$; 
and $s\beta$ is a $b$-adic rational modulo $1$ exactly when $\beta \in E_0$.  
So it suffices to treat the case $s=1$, with $s\beta$ in place of $\beta$.  We assume from now on that $s=1$, so that $E_0$ is the set of $b$-adic rationals modulo $1$.  

\emph{The digit average.}  Let $\rho (\theta) \coloneqq \Expectation_{d \in D} e(d\theta)$.  Then $\lvert \rho (\theta)\rvert \leq 1$, with equality only if $\theta \in \mathbb Z$.  
Indeed, equality forces the numbers $e(d\theta)$, $d \in D$, to coincide, hence to equal $1$, as $0 \in D$; 
so $d\theta \in \mathbb Z$ for every $d \in D$, and then $\theta \in \mathbb Z$, as $\gcd(D)=1$.  

\emph{Products over blocks.}  By Lemma \ref{l:Diff}(i), as $j$ runs over $0 \leq j < \delta ^{\ell}$, the base $b$ digits of $k_j$ in the positions $0 , \ldots, \ell-1$ 
run independently over $D$, and $k_j$ has no other non-zero digits.  Hence 
\begin{equation} \label{e:Pell}
P_\ell (\beta) \coloneqq \Expectation_{0\leq j<\delta ^{\ell}} e(\beta k_j) = \prod_{c=0}^{\ell-1} \rho (b ^{c}\beta) . 
\end{equation}

\emph{Window split.}  Fix $i \geq 1$ and an interval $[M,N)$, and let $M_i \coloneqq \lceil M/\delta ^{i}\rceil$ and $N_i \coloneqq \lfloor N/\delta ^{i}\rfloor$.  
At most $2\delta ^{i}$ indices of $[M,N)$ lie outside $[\delta ^{i}M_i, \delta ^{i}N_i)$; when $N_i \leq M_i$ this interval is empty, and $N-M < 2\delta ^{i}$.  
The remaining indices are partitioned into the blocks $\{\delta ^{i}m + j \colon 0 \leq j < \delta ^{i}\}$, $M_i \leq m < N_i$.  
By Lemma \ref{l:Diff}(ii), $e(\beta k_{\delta ^{i}m+j}) = e(\beta b ^{i}k_m)\, e(\beta k_j)$, so the sum over a block is $\delta ^{i}P_i(\beta)\, e(\beta b ^{i}k_m)$, of modulus at most $\delta ^{i}\lvert P_i(\beta)\rvert$.  
There are $N_i - M_i \leq (N-M)/\delta ^{i}$ blocks, and therefore 
\begin{equation} \label{e:windowSplitE}
\Bigl\lvert \Expectation_{M\leq n<N} e(\beta k_n) \Bigr\rvert \leq \lvert P_i(\beta)\rvert + \frac{2\delta ^{i}}{N-M} , \qquad M \geq 0 . 
\end{equation}
The bound does not depend on $M$; this is where the uniformity comes from.  

\emph{The limit of $P_\ell$.}  Write $\lVert \theta\rVert$ for the distance from $\theta$ to the nearest integer.  
Suppose first that $\lVert b ^{c}\beta\rVert \geq \eta$ for infinitely many $c$, for some $\eta>0$.  
As $\rho$ is continuous and $1$-periodic, and $\lvert \rho\rvert <1$ off the integers, there is a $\kappa>0$ with $\lvert \rho (\theta)\rvert \leq 1-\kappa$ whenever $\lVert \theta\rVert \geq \eta$.  
So infinitely many factors in \eqref{e:Pell} have modulus at most $1-\kappa$, all have modulus at most $1$, and $P_\ell(\beta) \to 0$.  

Otherwise $\lVert b ^{c}\beta\rVert \to 0$, and we show that $\beta \in E_0$.  Choose $c_0$ with $\lVert b ^{c}\beta\rVert < 1/(2b)$ for all $c \geq c_0$.  
If $0 < \lVert b ^{c}\beta\rVert < 1/(2b)$, write $b ^{c}\beta = z + \theta$ with $z \in \mathbb Z$ and $\lvert \theta\rvert = \lVert b ^{c}\beta\rVert$; 
then $b ^{c+1}\beta = bz + b\theta$ with $0<\lvert b\theta\rvert < 1/2$, so that $\lVert b ^{c+1}\beta\rVert = b \lVert b ^{c}\beta\rVert$.  
Hence if $\lVert b ^{c_0}\beta\rVert \neq 0$, the quantities $\lVert b ^{c_0+r}\beta\rVert = b ^{r}\lVert b ^{c_0}\beta\rVert$ would eventually exceed $1/(2b)$, a contradiction.  
So $b ^{c_0}\beta \in \mathbb Z$, that is, $\beta$ is a $b$-adic rational.  

\emph{Conclusion.}  Let $\beta \notin E_0$ and $\epsilon > 0$.  By the previous step, $P_i(\beta) \to 0$, so we may fix $i$ with $\lvert P_i(\beta)\rvert < \epsilon$.  
By \eqref{e:windowSplitE}, $\sup_{M\geq 0}\lvert \Expectation_{M\leq n<M+L} e(\beta k_n)\rvert \leq \epsilon + 2\delta ^{i}/L$, which is less than $2\epsilon$ once $L > 2\delta ^{i}/\epsilon$.  
As $\epsilon$ was arbitrary, this proves \eqref{e:irrationalUniform}.  
\end{proof}

The form in which Lemma \ref{l:irrational} enters the proof is the following.  For $\beta \in \mathbb R$ and $L \geq 1$ set 
\begin{equation} \label{e:Phi}
\Phi _L (\beta) \coloneqq \sup _{M\geq 0} \Bigl\lvert \Expectation_{M\leq n<M+L} e ( \beta k_n) \Bigr\rvert . 
\end{equation}

\begin{lemma}[Uniform spectral lemma] \label{l:uniformSpectral}
Let $\sigma$ be a finite positive Borel measure on $\mathbb T$ without atoms, and let $s \geq 1$ be an integer.  Then 
\begin{equation}
\lim_{L\to\infty} \int_{\mathbb T} \Phi _L (s \alpha ) \;d \sigma ( \alpha ) = 0 . 
\end{equation}
\end{lemma}

\begin{proof}
Each $\Phi_L$ is a supremum of countably many continuous functions, hence Borel, and $0 \leq \Phi _L \leq 1$.  
By Lemma \ref{l:irrational}, $\Phi_L(s \alpha ) \to 0$ unless $s \alpha \in E_0$ modulo $1$.  
Each point of $E_0$ has $s$ preimages under $\alpha \mapsto s\alpha$, so these $\alpha$ form a countable set, which is $\sigma$-null as $\sigma$ has no atoms.  
Dominated convergence completes the proof.  
\end{proof}

For instance, if $f \in L^2(X)$ has an atomless spectral measure, then, by \eqref{e:spectral}, 
\begin{equation} \label{e:Uniform}
\Bigl\lVert \Expectation _{M\leq n<N} T ^{s k_n}f \Bigr\rVert _2 ^2 
= \int_{\mathbb T} \Bigl\lvert \Expectation _{M\leq n<N} e(s\alpha k_n) \Bigr\rvert ^2 \;d \sigma _f (\alpha) 
\leq \int_{\mathbb T} \Phi _{N-M}(s \alpha )\;d \sigma _f ( \alpha ) , 
\end{equation}
which tends to zero as $N-M\to\infty$, uniformly in $M$.  This is the uniform ergodic theorem along $\mathcal C$, for such $f$.  

\subsection{Structure of the proof} 

The method of proof follows the lines of ergodic proofs of Szemer\'edi's Theorem. To summarize this approach, we say that a measure preserving system $(X,\mathcal A,m,T)$ is \emph{Cantor-Szemer\'edi},  or \emph{CSZ}, 
if this holds: 
For all measurable $A\subset X$, with $m(A) >0$, all integers $t \geq 1$, all integers $i \geq 0$, and all Cantor sets $\mathcal C = \mathcal C_{b,D}$, with base $b$, 
digit set $D$, and $0\in D$, there holds 
\begin{equation}  \label{e:CSZdef}
    \liminf_{N-M\to\infty} \Expectation_{M\leq n<N} 
    \int_X  \prod _{s=0} ^t T^{s b ^{i} k_n}\mathbf 1_A \;dm >0. 
\end{equation}
Here, $\mathcal C = \{0=k_0<k_1<\cdots\}$.  
The averages above are uniform in the index $n$; by Corollary \ref{c:syndetic}, 
this implies that the set of indices $n$ which contribute is syndetic, which is the point of insisting on the uniformity.  

It is equivalent to ask for \eqref{e:CSZdef} with $\mathbf 1_A$ replaced by an arbitrary non-negative $f \in L ^{\infty}(X)$ 
which is not identically zero.  One direction is trivial.  For the other, choose $c>0$ with $m(A)>0$ for $A \coloneqq \{f>c\}$; 
then $f \geq c \mathbf 1_A$, and all the factors are non-negative, so the average for $f$ is at least $c ^{t+1}$ times the average for $\mathbf 1_A$.  
In particular, the case $i=0$ of \eqref{e:CSZdef} is exactly the conclusion \eqref{e:CSZ} of Theorem \ref{t:CSZ}.  

Two comments on the dilation parameter $i$ are in order.  First, the case $i=0$ is the statement of Theorem \ref{t:CSZ}, 
and the other cases are statements of the same kind: since $b ^{i} k_n = k _{\delta ^{i} n}$ by Lemma \ref{l:Diff}(ii), 
the condition for a given $i$ is \eqref{e:CSZdef} with $i=0$, with the index restricted to the progression $\delta ^{i} \mathbb N_0$.  
Put another way, for a \emph{fixed} Cantor set $\mathcal C$ of base $b$, the condition for the dilation $i$ is the condition for $i=0$, 
for the system $(X,\mathcal A,m,T ^{b ^{i}})$ and the same $\mathcal C$.  
We stress that this is a statement about one Cantor set at a time.  It does not say that a CSZ system $(X,\mathcal A,m,T)$ gives rise to 
CSZ systems $(X,\mathcal A, m, T ^{b ^{i}})$: the CSZ property quantifies over all Cantor sets, and for a second Cantor set the available 
dilations are powers of \emph{its} base, not of $b$.  
Second, carrying the parameter $i$ is not a luxury: the proof for compact extensions in \S\ref{sec:compact} 
uses the hypothesis on the factor with a dilation determined by that proof, and returns the conclusion 
for the extension with the dilation it was given.  It is the self-similarity $b ^{i}\mathcal C \subset \mathcal C$ of the Cantor set which makes this  work.  

We prove 

\begin{theorem} \label{t:ECSZ}
    Every ergodic system satisfying the standing assumption of \S\ref{sub:mps} is CSZ.  
\end{theorem}

This proves Theorem \ref{t:CSZ}.  By Proposition \ref{p:reduction}, Theorem \ref{t:ECSZ} implies that every invertible measure preserving system is CSZ; 
the case $i=0$, with the passage from sets to non-negative functions recorded after \eqref{e:CSZdef}, is Theorem \ref{t:CSZ}.  
Theorem \ref{t:CSZ-Integers} then follows by the Furstenberg correspondence principle, 
as explained in the introduction.  

The proof of Theorem \ref{t:ECSZ} follows the maximal-factor strategy of Furstenberg, Katznelson and Ornstein \cite{MR670131}, 
in the form presented by Moreira \cite{MoreiraNotes}*{\S 8.1}.  
A \emph{factor} of $(X,\mathcal A,m,T)$ is a $T$-invariant sub-$\sigma$-algebra $\mathcal D \subset \mathcal A$; it is CSZ if $(X,\mathcal D,m,T)$ is; 
and factors which agree up to null sets are identified.  
The proof rests on three statements: the CSZ property passes to weakly mixing and to compact extensions (Lemma \ref{l:Extensions}); 
the CSZ factors are closed under suprema of chains (Lemma \ref{l:chain}); and every proper factor has a non-trivial extension of one of these two kinds 
(Lemma \ref{l:dichotomy}).  We first record what is used about powers of $T$.  

\begin{lemma} \label{l:powers}
Let $k \geq 1$ be an integer.  
\begin{enumerate}
\item If $\mathcal E$ is a compact extension of $\mathcal D$ for $T$, then it is one for $T ^{k}$. 
\item Let $f,g \in L^2(\mathcal E\mid \mathcal D)$ be such that 
\begin{equation} \label{e:wmHypothesis}
\lim_{N} \bigl\lVert \Expectation_N \lvert \langle T ^{n}f,g \rangle _{\mathcal E \mid \mathcal D} \rvert \bigr\rVert _{L^2(\mathcal D)} =0 . 
\end{equation}
Then the same holds along the progression $k \mathbb N$, that is 
\begin{equation} \label{e:wmProgression}
\lim_{N} \bigl\lVert \Expectation_N \lvert \langle T ^{kn}f,g \rangle _{\mathcal E \mid \mathcal D} \rvert \bigr\rVert _{L^2(\mathcal D)} =0 . 
\end{equation}
In particular, if every $f \in L ^{\infty}(\mathcal E)$ with $\mathbb E(f\mid \mathcal D)=0$ is conditionally weakly mixing for $T$, 
then the same holds for $T ^{k}$. 
\end{enumerate}
\end{lemma}

\begin{proof}
For (1), the orbit $\{T ^{kn}f \colon n \in \mathbb Z\}$ is a subset of $\{T^nf \colon n \in \mathbb Z\}$, so conditional compactness is inherited, 
and the conditionally compact functions therefore remain dense in $L^2(\mathcal E)$. 

For (2), set $\phi_n \coloneqq \lvert \langle T^nf,g\rangle _{\mathcal E\mid \mathcal D}\rvert$ and $a_n \coloneqq \lVert \phi _n\rVert _{L^2(\mathcal D)}$.  
By the conditional Cauchy--Schwarz inequality, 
$0 \leq \phi _n \leq \lVert T ^{n} f\rVert _{\mathcal E\mid \mathcal D} \lVert g\rVert _{\mathcal E\mid \mathcal D} \leq C$ for a constant $C$, 
since $f, g$ lie in the Hilbert module.  
As the $\phi_n$ are non-negative, \eqref{e:wmHypothesis} gives 
\begin{equation}
\Expectation_N \lVert \phi _n \rVert_{L^1(\mathcal D)} = \bigl\lVert \Expectation_N \phi _n \bigr\rVert _{L^1(\mathcal D)} 
\leq \bigl\lVert \Expectation_N \phi _n \bigr\rVert _{L^2(\mathcal D)} \to 0 , 
\end{equation}
and then $\Expectation_N a_n ^2 \leq C \,\Expectation_N \lVert \phi_n\rVert _{L^1(\mathcal D)} \to 0$, 
so that $\Expectation_N a_n \to 0$ by Cauchy--Schwarz.  
A bounded non-negative sequence has $\Expectation_N a_n \to 0$ if and only if $a_n \to 0$ along a set $S$ of density one.  
The complement of $S$ has density zero, so $\lvert S ^{c} \cap [1, kN]\rvert = o(N)$ for fixed $k$; 
that is, $a _{kn} \to 0$ along a set of $n$ of density one, so that $\Expectation_N a _{kn} \to 0$.  
Since $\lVert \Expectation_N \phi_{kn}\rVert _{L^2(\mathcal D)} \leq \Expectation_N a _{kn}$, this proves \eqref{e:wmProgression}.  
For the last statement, take $g=f$ with $\mathbb E (f\mid \mathcal D)=0$: then \eqref{e:wmHypothesis} is the definition of 
$f$ being conditionally weakly mixing, and \eqref{e:wmProgression} is that same statement for $T^k$. 
\end{proof}

\begin{lemma}[Extension principles] \label{l:Extensions}
Let $\mathcal D \subset \mathcal E$ be factors of $(X,\mathcal A,m,T)$, and suppose that $\mathcal D$ is CSZ.  
\begin{enumerate}
\item[(W)] If $\mathcal E$ is a weakly mixing extension of $\mathcal D$, then $\mathcal E$ is CSZ. 
\item[(K)] If $\mathcal E$ is a compact extension of $\mathcal D$, then $\mathcal E$ is CSZ. 
\end{enumerate}
\end{lemma}

Part (W) is proved in \S\ref{sec:wm}, as a consequence of Lemma \ref{l:wmExtension}; part (K) is Theorem \ref{t:compactExtension}.  
The trivial factor, consisting of the null and the conull sets, is CSZ: if $A$ is in it and $m(A)>0$, then $A$ is conull, 
and every integral in \eqref{e:CSZdef} equals $1$.  
A weakly mixing system is a weakly mixing extension of the trivial factor, and a compact system is a compact extension of it, 
so Lemma \ref{l:Extensions} shows in particular that weakly mixing systems and compact systems are CSZ; see Corollary \ref{t:wm}.  
Theorem \ref{t:compactSystem}, which treats a restricted class of compact systems by a direct argument, 
is not used in the proof of the main theorem.  

\begin{lemma}[Chains] \label{l:chain}
Let $\{\mathcal D_\alpha\}_{\alpha \in I}$ be a non-empty family of CSZ factors of $(X,\mathcal A,m,T)$ which is totally ordered by inclusion, 
and let $\mathcal D$ be the $\sigma$-algebra generated by their union.  Then $\mathcal D$ is a CSZ factor.  
\end{lemma}

\begin{proof}
The $\sigma$-algebra $\mathcal D$ is $T$-invariant, as each $\mathcal D_\alpha$ is.  

\emph{Approximation.}  For every $E \in \mathcal D$ and $\eta>0$ there is an $\alpha$ with 
$\lVert \mathbf 1_E - \mathbb E(\mathbf 1_E \mid \mathcal D_\alpha)\rVert_2 < \eta$; see \cite{MoreiraNotes}*{Lemma 8.1.5}.  
For completeness: the union $\bigcup_\alpha \mathcal D_\alpha$ is an algebra of sets, since any two of its members lie in a common $\mathcal D_\alpha$; 
and $\bigcup _\alpha L^2(\mathcal D_\alpha)$ is a linear subspace of $L^2(X)$, with closure $\mathcal H$, say.  
The sets $B \in \mathcal D$ with $\mathbf 1_B \in \mathcal H$ contain that algebra, and they form a monotone class, since monotone limits of indicators converge in $L^2$.  
By the monotone class theorem, they are all of $\mathcal D$.  So for some $\alpha$ there is a $g' \in L^2(\mathcal D_\alpha)$ with $\lVert \mathbf 1_E - g'\rVert_2 < \eta$, 
and $\mathbb E(\mathbf 1_E \mid \mathcal D_\alpha)$, the orthogonal projection of $\mathbf 1_E$ onto $L^2(\mathcal D_\alpha)$, is at least as close.  

\emph{A set in a single factor.}  Fix a Cantor set $\mathcal C$, integers $t \geq 1$ and $i \geq 0$, and $E \in \mathcal D$ with $m(E)>0$.  
Put $\kappa \coloneqq 1/(2(t+1))$.  Choose $\alpha$ such that $g \coloneqq \mathbb E(\mathbf 1_E \mid \mathcal D_\alpha)$ satisfies 
$\lVert \mathbf 1_E - g\rVert_2 ^2 < \kappa ^2 m(E)$, and let $H \coloneqq \{ g > 1-\kappa\} \in \mathcal D_\alpha$.  
Then $m(H)>0$: otherwise $1-g \geq \kappa$ almost everywhere on $E$, and so $\lVert \mathbf 1_E - g\rVert_2 ^2 \geq \kappa ^2 m(E)$.  

\emph{Comparison, for every step.}  Let $d$ be an integer, and $x \in H_d \coloneqq H \cap T ^{-d}H \cap \cdots \cap T ^{-td}H$.  
As $\mathcal D_\alpha$ is $T$-invariant, $\mathbb E(T ^{sd}\mathbf 1_E \mid \mathcal D_\alpha)(x) = g(T ^{sd}x) > 1-\kappa$ for $0 \leq s \leq t$.  
Since indicators take only the values $0$ and $1$, we have $\prod_{s=0}^t T ^{sd}\mathbf 1_E \geq 1 - \sum_{s=0}^t (1 - T ^{sd}\mathbf 1_E)$ pointwise, and so, on $H_d$, 
\begin{equation}
\mathbb E\Bigl( \prod_{s=0}^t T ^{sd}\mathbf 1_E \;\Big\vert\; \mathcal D_\alpha\Bigr) > 1 - (t+1) \kappa = \frac 12 . 
\end{equation}
As $H_d \in \mathcal D_\alpha$ and the product is non-negative, it follows that 
\begin{equation} \label{e:chainComparison}
m \bigl( E \cap T ^{-d}E \cap \cdots \cap T ^{-td}E \bigr) 
\geq \int_{H_d} \mathbb E\Bigl( \prod_{s=0}^t T ^{sd}\mathbf 1_E \;\Big\vert\; \mathcal D_\alpha\Bigr) dm 
\geq \frac 12 m(H_d) , \qquad d \in \mathbb Z . 
\end{equation}

\emph{Conclusion.}  Take $d = b ^{i}k_n$ in \eqref{e:chainComparison}, and average over $M \leq n < N$.  Since \eqref{e:chainComparison} holds for every $d$, 
\begin{equation}
\liminf_{N-M\to\infty} \Expectation_{M\leq n<N} m \Bigl( \bigcap_{s=0}^t T ^{-sb ^{i}k_n}E \Bigr) 
\geq \frac 12 \liminf_{N-M\to\infty} \Expectation_{M\leq n<N} m \Bigl( \bigcap_{s=0}^t T ^{-sb ^{i}k_n}H \Bigr) > 0 , 
\end{equation}
the last inequality because $H \in \mathcal D_\alpha$, $m(H)>0$, and $\mathcal D_\alpha$ is CSZ.  As $\mathcal C$, $t$, $i$ and $E$ were arbitrary, $\mathcal D$ is CSZ.  
\end{proof}

The comparison \eqref{e:chainComparison} is the one point at which the CSZ property enters, and it holds for every step $d$ separately.  
This is why the classical argument applies verbatim to the CSZ property, with its uniform averages over intervals of indices, 
its quantification over all Cantor sets, and its dilations: each of these is a statement about which steps $d$ are averaged.  

\begin{lemma}[Dichotomy] \label{l:dichotomy}
Let $(X,\mathcal A,m,T)$ be an ergodic system satisfying the standing assumption, and $\mathcal D \neq \mathcal A$ a factor.  
Then either $\mathcal A$ is a weakly mixing extension of $\mathcal D$, or there is a factor $\mathcal E$ with $\mathcal D \subsetneq \mathcal E \subset \mathcal A$ 
which is a compact extension of $\mathcal D$.  
\end{lemma}

This is due to Furstenberg \cite{MR0498471}; see \cite{MoreiraNotes}*{Proposition 8.6.1} for a proof in the present language.  
Our definitions differ slightly from those of \cite{MoreiraNotes}, and in the direction that makes the dichotomy apply as stated.  
There, $f$ is conditionally weakly mixing if $\lvert \langle T^nf,g\rangle_{\mathcal E\mid \mathcal D}\rvert \to 0$ in uniform Ces\`aro mean, in $L^2(\mathcal D)$, 
for \emph{every} $g$ in the Hilbert module \cite{MoreiraNotes}*{Definition 8.2.5}.  Our condition is the case $g=f$, with ordinary Ces\`aro means, and so is implied by it.  
Hence a weakly mixing extension in the sense of \cite{MoreiraNotes} is one in ours.  
Our definition of conditional compactness, with orbits indexed by $n \in \mathbb Z$ and approximants in the Hilbert module, is that of \cite{MoreiraNotes}*{Definition 8.2.8}.  
The proof in \cite{MoreiraNotes} is for ergodic systems on compact metric spaces, which is the setting of Theorem \ref{t:ECSZ}, and uses the disintegration of $m$ over $\mathcal D$.  

\begin{proof}[Proof of Theorem \ref{t:ECSZ}]
Let $\mathbf K$ be the set of CSZ factors of $(X,\mathcal A,m,T)$, partially ordered by inclusion.  
It is non-empty, as it contains the trivial factor; and every non-empty chain in $\mathbf K$ has an upper bound in $\mathbf K$, namely the $\sigma$-algebra generated by its union, by Lemma \ref{l:chain}.  
By Zorn's lemma, $\mathbf K$ has a maximal element $\mathcal D$.  

Suppose that $\mathcal D \neq \mathcal A$, and apply Lemma \ref{l:dichotomy}.  
If $\mathcal A$ is a weakly mixing extension of $\mathcal D$, then $\mathcal A$ is CSZ by Lemma \ref{l:Extensions}(W).  
Otherwise there is a compact extension $\mathcal E$ of $\mathcal D$ with $\mathcal D \subsetneq \mathcal E$, and $\mathcal E$ is CSZ by Lemma \ref{l:Extensions}(K).  
Either way, $\mathbf K$ contains a factor strictly larger than $\mathcal D$, contradicting maximality.  
Hence $\mathcal D = \mathcal A$, which is to say that the system is CSZ.  
\end{proof}

Finally, we record the van der Corput inequality, which is the starting point of the argument in \S\ref{sec:wm}. 

\begin{lemma}[van der Corput Inequality]  \label{vdC}
Let $u_M, \ldots, u_{N-1}$ be vectors in a Hilbert space, of norm at most $1$, and let $1\leq H < N-M$.  Then 
\begin{equation}
    \label{vdc} 
\Bigl\lVert \Expectation_{M\leq n <N} u_n \Bigr\rVert ^2 
\leq 
4\, \Expectation_{1\leq h\leq H} 
\Bigl\lvert 
\Expectation_{M\leq n <N-h} \langle u_{n+h}, u_n\rangle  
\Bigr\rvert
+ \frac 2H . 
\end{equation}
\end{lemma}

\begin{proof}
Write $L = N-M$, and set $u_n = 0$ for $n \notin [M,N)$.  Each $u_n$ with $n \in [M,N)$ occurs exactly $H$ times among the vectors $u_{m+h}$, 
$1 \leq h \leq H$, $m \in I \coloneqq [M-H, N-1)$, an interval of fewer than $L+H$ integers.  By Cauchy--Schwarz in $m$, 
\begin{equation}
\Bigl\lVert \sum_{n} u_n \Bigr\rVert ^2 
= \Bigl\lVert \sum _{m \in I} \frac 1H \sum _{h=1}^H u_{m+h} \Bigr\rVert ^2 
\leq \frac{L+H}{H^2} \sum _{h,h'=1}^H \sum _{m \in I} \langle u_{m+h}, u_{m+h'} \rangle . 
\end{equation}
The $H$ diagonal terms contribute at most $L$ each.  For $r = h-h' \neq 0$, the inner sum is 
$S(r) \coloneqq \sum _{n} \langle u_{n+r}, u_n\rangle$, with $S(-r) = \overline{S(r)}$, and exactly $H - \lvert r\rvert$ pairs $(h,h')$ have $h-h'=r$.  
Moreover $\lvert S(r)\rvert = (L-r) \lvert c_r\rvert \leq L \lvert c_r\rvert$ for $1 \leq r <H$, where $c_r \coloneqq \Expectation_{M\leq n<N-r}\langle u_{n+r},u_n\rangle$.  Hence 
\begin{equation}
\Bigl\lVert \Expectation_{M\leq n<N} u_n \Bigr\rVert ^2 
\leq \frac{L+H}{HL} + \frac{2(L+H)}{H^2 L} \sum _{r=1}^{H-1} (H-r) \lvert c_r\rvert 
\leq \frac 1H + \frac 1L + \frac{2(L+H)}{L}\, \Expectation_{1\leq r\leq H} \lvert c_r\rvert . 
\end{equation}
As $H < L$, this is \eqref{vdc}. 
\end{proof}

\subsection{Reduction to ergodic systems on compact metric spaces} \label{sub:reduction}

The CSZ property asks for a positive liminf over all intervals of indices, which is more than the usual multiple recurrence statement, 
so we give the reduction to ergodic systems in full.  

\begin{proposition} \label{p:reduction}
Suppose that every ergodic system satisfying the standing assumption of \S\ref{sub:mps} is CSZ.  Then every invertible measure preserving system is CSZ.  
\end{proposition}

\begin{proof}
Fix an invertible measure preserving system $(X,\mathcal A,m,T)$, a Cantor set $\mathcal C$ of base $b$, integers $t \geq 1$ and $i \geq 0$, and $A \in \mathcal A$ with $m(A)>0$.  

\emph{A compact metric model.}  Let $\Omega \coloneqq \{0,1\} ^{\mathbb Z}$, a compact metric space in the product topology, with the shift $(\tau\omega)(n) \coloneqq \omega(n+1)$, a homeomorphism.  
Define $\Phi \colon X \to \Omega$ by $\Phi(x)(n) \coloneqq \mathbf 1_A(T ^{n}x)$.  Each coordinate of $\Phi$ is measurable, so $\Phi$ is measurable, and $\Phi\circ T = \tau \circ \Phi$.  
Hence $\mu \coloneqq \Phi_\ast m$ is a $\tau$-invariant Borel probability measure on $\Omega$.  Let $\widetilde A \coloneqq \{\omega \colon \omega(0) = 1\}$.  
Then $\Phi ^{-1}(\tau ^{-d}\widetilde A) = T ^{-d}A$ for every $d \in \mathbb Z$, and therefore 
\begin{equation} \label{e:model}
m \Bigl( \bigcap_{s=0}^t T ^{-d_s}A \Bigr) = \mu \Bigl( \bigcap_{s=0}^t \tau ^{-d_s}\widetilde A \Bigr) , \qquad d_0 , \ldots, d_t \in \mathbb Z . 
\end{equation}
In particular $\mu(\widetilde A) = m(A)>0$.  By \eqref{e:model}, the averages in \eqref{e:CSZdef} for $A$ in $(X,\mathcal A,m,T)$ coincide with those for $\widetilde A$ in $(\Omega,\mathcal B,\mu,\tau)$, 
with $\mathcal B$ the Borel $\sigma$-algebra of $\Omega$.  So it suffices to prove that the latter have a positive liminf.  

\emph{Ergodic decomposition.}  By \cite{MoreiraNotes}*{Theorem 7.3.1}, for $\mu$-almost every $y \in \Omega$ there is an ergodic $\tau$-invariant Borel probability measure $\mu_y$ on $\Omega$ 
such that, for every Borel set $B$, the map $y \mapsto \mu_y(B)$ is measurable and $\mu(B) = \int \mu_y(B)\;d\mu(y)$.  
For a $\tau$-invariant Borel probability measure $\nu$, write 
$c_n(\nu) \coloneqq \nu \bigl( \bigcap_{s=0}^t \tau ^{-sb ^{i}k_n}\widetilde A \bigr)$, so that $c_n(\mu) = \int c_n(\mu_y)\;d\mu(y)$ for every $n$.  
As $0 < \mu(\widetilde A) = \int \mu_y(\widetilde A)\;d\mu(y)$, the set $Y_+ \coloneqq \{y \colon \mu_y(\widetilde A)>0\}$ has positive measure.  
For $y \in Y_+$, the system $(\Omega,\mathcal B,\mu_y,\tau)$ is ergodic, invertible, and on a compact metric space, so it is CSZ by hypothesis, and 
\begin{equation}
c(y) \coloneqq \liminf_{N-M\to\infty} \Expectation_{M\leq n<N} c_n(\mu_y) > 0 , \qquad y \in Y_+ . 
\end{equation}

\emph{Passing the liminf through the integral.}  Let $c \coloneqq \liminf_{N-M\to\infty} \Expectation_{M\leq n<N} c_n(\mu)$, 
and choose intervals $[M_k,N_k)$, with $N_k - M_k \to \infty$, along which $\Expectation_{M_k\leq n<N_k} c_n(\mu) \to c$.  
The function $\ell(y) \coloneqq \liminf_k \Expectation_{M_k\leq n<N_k} c_n(\mu_y)$ is measurable, and $\ell(y) \geq c(y)$, 
because a liminf along one sequence of intervals whose lengths tend to infinity is at least the liminf over all such intervals.  
By Fatou's lemma, the integrands being non-negative, 
\begin{equation}
c = \lim_k \int \Expectation_{M_k\leq n<N_k} c_n(\mu_y)\;d\mu(y) \geq \int \ell(y)\;d\mu(y) \geq \int_{Y_+} \ell(y)\;d\mu(y) > 0 , 
\end{equation}
the last inequality because $\ell(y) \geq c(y) > 0$ for $y \in Y_+$, a set of positive measure.  As $\mathcal C$, $t$, $i$ and $A$ were arbitrary, $(X,\mathcal A,m,T)$ is CSZ.  
\end{proof}

The step through Fatou's lemma is the only point at which the uniformity of the CSZ property makes a difference here.  
A uniform liminf is not a liminf along a single sequence, so Fatou's lemma cannot be applied to it directly; 
it is applied instead along one sequence of intervals that realizes the liminf for $\mu$.  
Compare \cite{MoreiraNotes}*{\S 7.5}, where the corresponding reduction is carried out for uniform limits.  

We record where the hypotheses of the standing assumption, and ergodicity, are used.  
The compact metric structure is used for three things: the ergodic decomposition in Proposition \ref{p:reduction}; 
the disintegration of $m$ over a factor, behind the joining $\lambda$ of Lemma \ref{l:atomless} and so behind \S\ref{sec:wm}; 
and the dichotomy, Lemma \ref{l:dichotomy}, whose proof in \cite{MoreiraNotes} uses the same disintegration.  
The approximation step in the proof of Lemma \ref{l:chain} does not use it.  
Ergodicity is used only in the dichotomy.  Lemmas \ref{l:Extensions} and \ref{l:chain}, and Theorems \ref{t:compactSystem} and \ref{t:compactExtension}, hold without it.

\section{Weak Mixing Extensions} \label{sec:wm}

We show that the CSZ property lifts through weakly mixing extensions.  Weakly mixing systems are the case of the trivial factor.
Throughout, let $\mathcal C = \mathcal C_{b,D}=\{0=k_0<k_1<\cdots\}$ be a classical Cantor set, with $0\in D$, and set $\delta = \lvert D\rvert$ as in \eqref{e:delta}.
All averages are uniform in the index, as in the definition \eqref{e:CSZdef} of the CSZ property.

The results are stated for an increasing tuple of positive integer exponents $0<s_1 < \cdots < s_t$, rather than for the exponents $1,2 , \ldots, t$.
This costs nothing, and the induction requires it: the reduction to mean zero factors passes to a \emph{subset} of the exponents,
and the van der Corput step replaces the exponents $s_j$ by the differences $s_j - s_1$, which are again positive and increasing.
The generality also absorbs the dilation parameter $i$ of \eqref{e:CSZdef}, which is the case $s_j = j b ^{i}$.

Let $\mathcal D \subset \mathcal E$ be $T$-invariant sub-$\sigma$-algebras of $\mathcal A$.  We will use the hypothesis
\begin{equation} \label{e:WM}
\textup{every $f \in L ^{\infty}(\mathcal E)$ with $\mathbb E(f\mid \mathcal D) = 0$ is conditionally weakly mixing for $T$.}
\end{equation}
If $\mathcal E$ is a weakly mixing extension of $\mathcal D$, then \eqref{e:WM} holds, by definition;
and if \eqref{e:WM} holds for $T$, then it holds for every power $T ^{k}$, $k \geq 1$, by Lemma \ref{l:powers}(2).

\begin{lemma} \label{l:wmExtension}
Let $\mathcal D \subset \mathcal E$ be $T$-invariant sub-$\sigma$-algebras satisfying \eqref{e:WM},
let $0<s_1 < \cdots < s_t$ be integers, and let $f_1 ,\ldots, f_t\in L ^{\infty}(\mathcal E)$.   Then
\begin{equation} \label{e:wmConclusion}
  \lim_{N-M\to\infty}
  \biggl \lVert  \Expectation_{M\leq n<N}  \Bigl[ \prod _{j=1}^t
  T^{s_jk_n}f_j
  -  \prod _{j=1}^t
  T^{s_jk_n}\mathbb{E} (f_j \mid \mathcal D) \Bigr] \biggr\rVert _{L^2(X)} =0 .
  \end{equation}
\end{lemma}

Since $\lVert\, \lVert g\rVert _{\mathcal E\mid \mathcal D}\rVert _{L^2(\mathcal D)} = \lVert g\rVert _{L^2(X)}$, this is convergence in the conditional norm, in $L^2(\mathcal D)$.
The Lemma proves part (W) of Lemma \ref{l:Extensions}: if $\mathcal D$ is CSZ and $\mathcal E$ is a weakly mixing extension of it, then $\mathcal E$ is CSZ.
Indeed \eqref{e:WM} holds, since a weakly mixing extension satisfies it for every $f$ in the Hilbert module, which contains $L^\infty(\mathcal E)$.
Indeed, take $s_j = jb ^{i}$, take each $f_j = \mathbf 1_A$ with $A \in \mathcal E$ and $m(A)>0$, and set $g \coloneqq \mathbb E (\mathbf 1_A \mid \mathcal D)$.
Pairing \eqref{e:wmConclusion} against $\mathbf 1_A$, and using that $\prod_j T ^{jb ^{i}k_n}g$ is $\mathcal D$ measurable, so that its integral against $\mathbf 1_A$ equals its integral against $g$,
\begin{equation}
\Expectation_{M\leq n<N} \int _X \prod _{s=0}^t T ^{sb ^{i}k_n}\mathbf 1_A \;dm
= \Expectation_{M\leq n<N} \int _X \prod _{s=0}^t T ^{sb ^{i}k_n} g \;dm + o(1) ,
\end{equation}
as $N-M\to\infty$.  The liminf of the right side is positive because $\mathcal D$ is CSZ, in the form recorded after \eqref{e:CSZdef} for non-negative functions;
note that $g \geq 0$ and $\int g \;dm = m(A)>0$.

\begin{corollary} \label{t:wm}  Let  $(X,\mathcal A, m,T)$ be a weakly mixing system, let $0<s_1 < \cdots < s_t$ be integers,
 and let  $f_1 , \ldots, f_t \in L^\infty (X)$.  Then
 \begin{equation}
  \lim_{N-M\to\infty}  \Expectation_{M\leq n<N}  \prod _{j=1}^t
  T^{s_j k_n}f_j =\prod _{j=1}^t \int f_j \;dm,
  \end{equation}
  the limit holding in $L^2( X)$ norm.
\end{corollary}

\begin{proof}
Apply Lemma \ref{l:wmExtension} with $\mathcal D$ the trivial $\sigma$-algebra and $\mathcal E = \mathcal A$.
Then $\mathbb E(f \mid \mathcal D) = \int f \;dm$ is constant, the conditional inner product is the inner product of $L^2(X)$,
and \eqref{e:WM} is the case $g=f$ of the definition \eqref{wm-functions} of weak mixing.
\end{proof}

So a weakly mixing system is CSZ, which is (W) for the trivial factor: taking $s_j = jb ^{i}$ and $f_1 = \cdots = f_t = \mathbf 1_A$,
and pairing against $\mathbf 1_A$, the averages in \eqref{e:CSZdef} converge to $m(A) ^{t+1}>0$.

\begin{proof}[Proof of Lemma \ref{l:wmExtension}]

\noindent\textbf{Reduction.}
Rescaling, we may assume $\lvert f_j\rvert \leq 1$ for all $j$.  By the telescoping identity
\begin{multline}
\prod _{j=1}^t T^{s_jk_n}f_j - \prod _{j=1}^t T^{s_jk_n}\mathbb{E} (f_j \mid \mathcal D) \\
= \sum _{j_0=1}^t \prod _{j<j_0} T ^{s_jk_n} \mathbb E(f_j\mid \mathcal D) \cdot T ^{s_{j_0}k_n} \bigl( f_{j_0} - \mathbb E(f_{j_0}\mid \mathcal D)\bigr) \cdot \prod _{j>j_0} T ^{s_jk_n}f_j ,
\end{multline}
each term on the right is a product of the same form, with bounded functions, in which the $j_0$-th function has conditional expectation zero.
After rescaling once more, it therefore suffices to prove that, if $\lvert f_j\rvert \leq 1$ for all $j$ and $\mathbb E (f_{j_0}\mid \mathcal D)=0$ for some $j_0$, then 
\begin{equation} \label{e:star}
\lim_{N-M\to\infty} \Bigl\lVert \Expectation_{M\leq n<N} \prod_{j=1}^t T ^{s_jk_n}f_j \Bigr\rVert _{L^2(X)} = 0 . 
\end{equation}
We prove the Lemma by induction on $t$, simultaneously for all invertible measure preserving transformations of $(X,\mathcal A,m)$ for which $\mathcal D$ and $\mathcal E$ are invariant and \eqref{e:WM} holds,
and all increasing tuples of positive exponents.  At each $t$ we prove \eqref{e:star}, which gives the Lemma for that $t$ by the reduction just made.

\medskip
\noindent\textbf{The case $t=1$.}
Here $f = f_1$ has $\mathbb E(f\mid \mathcal D)=0$, so it is conditionally weakly mixing by \eqref{e:WM}, and its spectral measure $\sigma _f$ has no atoms by Lemma \ref{l:atomless}(i).
By \eqref{e:Uniform} and the uniform spectral Lemma \ref{l:uniformSpectral},
$\lVert \Expectation_{M\leq n<N} T ^{s_1k_n}f\rVert _2 ^2 \leq \int \Phi _{N-M}(s_1 \alpha )\;d \sigma _f ( \alpha ) \to 0$.

\medskip
\noindent\textbf{The inductive step: choice of parameters.}
Let $t \geq 2$, and assume the Lemma for all tuples of fewer than $t$ exponents.
Fix $f_1 , \ldots, f_t$ with $\lvert f_j\rvert \leq 1$ and $\mathbb E(f_{j_0}\mid \mathcal D) = 0$, and fix $0< \epsilon <1$.
Write $f \coloneqq f_{j_0}$ and $s \coloneqq s_{j_0}$, and for $m \in \mathbb Z$ set
\begin{equation} \label{e:phipsi}
\phi (m) \coloneqq \bigl\lVert \langle T ^{m}f, f\rangle _{\mathcal E\mid \mathcal D}\bigr\rVert _{L ^{1}(\mathcal D)},
\qquad
\psi (m) \coloneqq \bigl\lVert \langle T ^{m}f, f\rangle _{\mathcal E\mid \mathcal D}\bigr\rVert _{L ^{2}(\mathcal D)} ^2 .
\end{equation}
Both are non-negative, and $\phi (m) ^2 \leq \psi (m) \leq 1$.
By Lemma \ref{l:atomless}(ii) and \eqref{e:spectral}, $\psi (m) = \int e(m \alpha )\;d \sigma ( \alpha )$ for a finite positive measure $\sigma = \sigma _G$ without atoms.
By Lemma \ref{l:uniformSpectral}, we may fix $H \geq 1/ \epsilon$ with
\begin{equation} \label{e:Hchoice}
\int_{\mathbb T} \Phi _H (s \alpha ) \;d \sigma ( \alpha ) < \epsilon .
\end{equation}
Then fix $\ell$ with $\delta ^{\ell} > H/\epsilon$.
Put $u_n \coloneqq \prod _{j=1}^t T ^{s_jk_n}f_j$, so that $\lVert u_n\rVert _2 \leq 1$.

\medskip
\noindent\textbf{Step 1: one shift at a time.}
We claim that there is an $L_0$ such that, for every $1 \leq h \leq H$ and every interval $[M',N')$ of length at least $L_0$,
\begin{equation} \label{e:perShift}
\Bigl\lvert \Expectation_{M'\leq n<N'} \langle u_{n+h}, u_n\rangle \Bigr\rvert
\leq \Expectation_{M'\leq n<N'} \phi \bigl( s(k_{n+h}-k_n)\bigr) + 3 \epsilon .
\end{equation}

Fix $h$.  For $0 \leq c < \delta ^{\ell}-h$, let $P_c$ and $\Delta _c$ be as in Lemma \ref{l:differences}, so that $k_{n+h}-k_n = \Delta _c$ for $n \in P_c$.
For such $n$, writing $\Delta = \Delta _c$, we have $T ^{s_jk_{n+h}} = T ^{s_jk_n}T ^{s_j\Delta}$, and so
\begin{equation} \label{e:2inner}
\langle u_{n+h}, u_n\rangle
= \int_X \prod _{j=1}^t T ^{s_jk_n} F_j \;dm
= \int_X F_1 \cdot \prod _{j=2}^t T ^{(s_j-s_1)k_n} F_j \;dm ,
\qquad F_j \coloneqq \overline{f_j} \cdot T ^{s_j\Delta _c}f_j ,
\end{equation}
the second equality by applying $T ^{-s_1k_n}$, which preserves $m$.  Note that $\lvert F_j\rvert \leq 1$.
Write $n = \delta ^{\ell}a + c$.  By Lemma \ref{l:Diff}(ii), $k_n = b ^{\ell}k_a + k_c$, and therefore
\begin{equation} \label{e:reindex}
\prod _{j=2}^t T ^{(s_j-s_1)k_n} F_j
= \prod _{j=2}^t \bigl( T ^{b ^{\ell}}\bigr) ^{(s_j-s_1)k_a} \bigl( T ^{(s_j-s_1)k_c} F_j \bigr) .
\end{equation}
The transformation $T ^{b ^{\ell}}$ satisfies \eqref{e:WM}, and the $t-1$ exponents $s_j - s_1$, $2 \leq j \leq t$, are positive and increasing.
So the induction hypothesis applies to the right side of \eqref{e:reindex}.
Since $\mathcal D$ is $T$-invariant, $\mathbb E( T ^{(s_j-s_1)k_c} F_j \mid \mathcal D) = T ^{(s_j-s_1)k_c}\mathbb E(F_j \mid \mathcal D)$; undoing \eqref{e:reindex},
and noting that as $n$ runs over $P_c \cap [M',N')$ the index $a$ runs over an interval of length at least $(N'-M')/ \delta ^{\ell} - 1$, we obtain
\begin{equation} \label{e:IH}
\lim_{N'-M'\to\infty}
\Bigl\lVert \Expectation_{n \in P_c\cap[M',N')} \Bigl[ \prod _{j=2}^t T ^{(s_j-s_1)k_n} F_j - \prod _{j=2}^t T ^{(s_j-s_1)k_n} \mathbb E(F_j\mid \mathcal D) \Bigr] \Bigr\rVert _{L^2(X)} = 0 .
\end{equation}
The shifts $T ^{(s_j-s_1)k_n}$ remain on the conditional expectations in \eqref{e:IH}: these are $\mathcal D$-measurable functions, not constants.
There are finitely many pairs $(h,c)$, so we may fix $L_0 \geq \delta ^{\ell}/ \epsilon$ such that the norm in \eqref{e:IH} is less than $\epsilon$ for all of them,
whenever $N'-M' \geq L_0$.

Pair \eqref{e:IH} with $F_1$, which is bounded by $1$, and use that the second product in \eqref{e:IH} is $\mathcal D$-measurable,
so that its integral against $F_1$ equals its integral against $\mathbb E(F_1\mid \mathcal D)$.  By \eqref{e:2inner},
\begin{equation}
\Bigl\lvert \Expectation_{n \in P_c\cap[M',N')} \langle u_{n+h}, u_n\rangle \Bigr\rvert
\leq \epsilon + \Expectation_{n \in P_c\cap[M',N')} \Bigl\lvert \int_X \prod _{j=1}^t T ^{(s_j-s_1)k_n} \mathbb E(F_j \mid \mathcal D) \;dm \Bigr\rvert .
\end{equation}
All $t$ factors in the last integral are bounded by $1$ in modulus.  Bounding all but the $j_0$-th by $1$, and using that $T$ preserves $m$,
the integral is at most $\lVert \mathbb E(F_{j_0}\mid \mathcal D)\rVert _{L ^{1}}$.
As $\mathbb E(F_{j_0}\mid \mathcal D) = \mathbb E( T ^{s\Delta _c}f \cdot \overline f \mid \mathcal D) = \langle T ^{s\Delta _c}f, f\rangle _{\mathcal E\mid \mathcal D}$,
this is $\phi (s\Delta _c) = \phi \bigl(s(k_{n+h}-k_n)\bigr)$ for every $n \in P_c$.  Hence
\begin{equation} \label{e:perClass}
\Bigl\lvert \Expectation_{n \in P_c\cap[M',N')} \langle u_{n+h}, u_n\rangle \Bigr\rvert
\leq \epsilon + \Expectation_{n \in P_c\cap[M',N')} \phi \bigl(s(k_{n+h}-k_n)\bigr) .
\end{equation}

Now write the average over $[M',N')$ as the weighted sum of the averages over the sets $P_c \cap [M',N')$, with weights $\lvert P_c\cap[M',N')\rvert/(N'-M')$,
plus the contribution of the remaining indices.  By Lemma \ref{l:differences}(ii), there are at most
$h \delta ^{-\ell}(N'-M') + \delta ^{\ell} \leq \epsilon (N'-M') + \delta ^{\ell}$ remaining indices, and for them $\lvert \langle u_{n+h},u_n\rangle\rvert \leq 1$.
The triangle inequality, \eqref{e:perClass}, and the non-negativity of $\phi$, which lets us extend its average to all of $[M',N')$, give
\begin{equation}
\Bigl\lvert \Expectation_{M'\leq n<N'} \langle u_{n+h}, u_n\rangle \Bigr\rvert
\leq \epsilon + \Expectation_{M'\leq n<N'} \phi \bigl(s(k_{n+h}-k_n)\bigr) + \epsilon + \frac{\delta ^{\ell}}{N'-M'} ,
\end{equation}
which is \eqref{e:perShift}, as $N'-M' \geq L_0 \geq \delta ^{\ell}/\epsilon$.  This proves the claim.

\medskip
\noindent\textbf{Step 2: averaging over the shifts.}
Let $N - M \geq L_0 + 2H$.  For each $1\leq h \leq H$ apply \eqref{e:perShift} to the interval $[M, N-h)$, whose length is at least $L_0$.  Averaging over $h$,
\begin{equation} \label{e:shiftAverage}
\Expectation_{1\leq h\leq H} \Bigl\lvert \Expectation_{M\leq n<N-h} \langle u_{n+h}, u_n\rangle \Bigr\rvert
\leq \Expectation_{1\leq h\leq H} \Expectation_{M\leq n<N-h} \phi \bigl(s(k_{n+h}-k_n)\bigr) + 3 \epsilon .
\end{equation}
The left side is exactly the quantity in the van der Corput inequality \eqref{vdc}: the absolute value is inside the average over $h$.
This is why each shift was treated separately in Step 1, and the shifts are averaged only now, after the correlation has been bounded by the non-negative function $\phi$.

By the Cauchy--Schwarz inequality and $\phi ^2 \leq \psi$, the double average on the right of \eqref{e:shiftAverage} is at most
$\bigl( \Expectation_{h} \Expectation_{M\leq n<N-h} \psi (s(k_{n+h}-k_n)) \bigr) ^{1/2}$.
Since $\psi \geq 0$ and $N-M \geq 2H$, enlarging each window $[M,N-h)$ to $[M,N)$ costs at most a factor $(N-M)/(N-M-H) \leq 2$.
Then, by the representation of $\psi$ in terms of $\sigma$, and since the left side below is real and non-negative, and so equal to its modulus, 
\begin{align}
\Expectation_{1\leq h\leq H} \Expectation_{M\leq n<N} \psi \bigl(s(k_{n+h}-k_n)\bigr)
&= \Bigl\lvert \int_{\mathbb T} \Expectation_{M\leq n<N} \Bigl[ e(-s\alpha k_n) \Expectation_{1\leq h\leq H} e(s \alpha k_{n+h}) \Bigr] d \sigma ( \alpha ) \Bigr\rvert \\
&\leq \int_{\mathbb T} \Expectation_{M\leq n<N} \Bigl\lvert \Expectation_{1\leq h\leq H} e(s \alpha k_{n+h}) \Bigr\rvert \; d \sigma ( \alpha ) \\
&\leq \int_{\mathbb T} \Phi _H(s \alpha ) \;d \sigma ( \alpha ) < \epsilon .
\end{align}
The first inequality takes the modulus inside the integral and the average over $n$, which is the direction the triangle inequality allows; 
the absolute value of the average over $h$ is kept intact.  For the second, note that for each fixed $n$ the average over $h$ is an average of $e(s\alpha k_u)$ 
over the interval of indices $u \in [n+1, n+H]$, of length $H$, so that its modulus is at most $\Phi _H(s \alpha )$, uniformly in $n$.  The last step is \eqref{e:Hchoice}.
Combining, the left side of \eqref{e:shiftAverage} is at most $\sqrt{2 \epsilon } + 3 \epsilon$.

\medskip
\noindent\textbf{Step 3: conclusion.}
By the van der Corput inequality, Lemma \ref{vdC}, and $1/H \leq \epsilon$, for $N - M \geq L_0 + 2H$ we have
\begin{equation}
\Bigl\lVert \Expectation_{M\leq n<N} u_n \Bigr\rVert _2 ^2 \leq 4 \bigl( \sqrt{2 \epsilon } + 3 \epsilon \bigr) + 2 \epsilon \leq 20 \sqrt \epsilon .
\end{equation}
As $\epsilon$ was arbitrary, this proves \eqref{e:star} for $t$, and so completes the induction.
\end{proof}

\section{Van der Waerden's Theorem} \label{sec:vdW}

The Cantor set version of van der Waerden's theorem proved in this section is not needed for the main theorem: 
the compact extension argument of \S\ref{sec:compact} does not use it.  We include it as a result of independent interest, 
and because it is the natural topological counterpart of Theorem \ref{t:CSZ}.  

We let $\mathcal C=\mathcal{C}_{b,D}$ be an integer Cantor set in base $b$ and digit set $D\subset \{0,\ldots ,b-1\}$, with $0\in D$.  
Recall that $\mathcal C$ is closed under multiplication by powers of $b$, so that 
\begin{equation} \label{e:bhC}
b ^{h} \mathcal C \subset \mathcal C , \qquad h \geq 0 . 
\end{equation}
By a \emph{topological dynamical system} we mean a pair $(X,T)$ with $X$ a compact metric space and $T \colon X \to X$ continuous. 

\begin{theorem}\label{thmCvdW} Let $(X,T)$ be a topological dynamical system, and $(U_\alpha)_{\alpha\in A}$ an open cover of $X$. Then for every $\ell \geq 1$ and every $h\geq 0$, there is an open set $U_\alpha$ which contains an arithmetic progression 
$x,T^rx,T^{2r}x,\cdots ,T^{(\ell-1)r}x$ for some $x\in X$ and some $r\in b ^{h}\mathcal C\setminus \{0\}$.
\end{theorem}

This result is a corollary to work of Bergelson and Leibman \cite{MR1325795}*{Theorem 0.13}. We include a proof here for the sake of completeness.   
Indeed, we follow the familiar lines of the proof of van der Waerden's Theorem.  
The auxiliary parameter $h$ costs us nothing, and it is what makes the induction below close: it permits us to insist that the 
differences produced at a given stage are supported on base $b$ digits beyond those already used.  

The form in which we use the theorem is the following finitary statement.  

\begin{corollary} \label{c:vdW}
For all integers $L, \ell \geq 1$ there is an integer $W= W(L,\ell )$ so that for every map 
$\chi \colon [W] \to \{1 , \ldots, L\}$ there are $w_0 \in [W]$ and $r \in \mathcal C \setminus \{0\}$ with 
$w_0 + ( \ell -1)r \leq W$ and $\chi (w_0) = \chi (w_0 + r) = \cdots = \chi (w_0 + (\ell -1) r)$. 
\end{corollary}

\begin{proof}
We first argue that every $\chi \colon \mathbb N \to \{1 , \ldots, L\}$ admits a monochromatic $\ell$-term progression with common difference in $\mathcal C\setminus \{0\}$.  
Let $\Omega = \{1 , \ldots, L\} ^{\mathbb N}$ with the product topology, a compact metric space, and let $T$ be the one-sided shift, 
$(T \eta )(n) = \eta (n+1)$, which is continuous.  Let $X \subset \Omega$ be the closure of the forward orbit $\{ T ^{m} \chi \colon m \geq 0\}$, 
a compact set with $TX \subset X$.  We use the one-sided shift precisely so that the progression we produce lies in $\mathbb N$.  
Let $U_c = \{ \eta \in X \colon \eta (1)=c\}$, for $1 \leq c \leq L$.  
These are open and closed, and cover $X$. By Theorem \ref{thmCvdW}, with $h=0$, there are $c$, $x \in X$ and $r \in \mathcal C\setminus\{0\}$ 
with $T ^{ir}x \in U_c$ for $0 \leq i \leq \ell -1$; that is, $x(1+ir)=c$ for $0 \leq i \leq \ell-1$.  
These conditions involve only the coordinates of $x$ in $[1, 1+(\ell-1)r]$, and $x$ is a limit of the points $T ^{m} \chi$, $m \geq 0$.  
So there is an $m \geq 0$ with $(T ^{m}\chi )(1+ir) = c$ for $0 \leq i \leq \ell-1$, that is, with 
$\chi (w_0 + ir) = c$ for $w_0 \coloneqq m+1 \geq 1$.  

The uniform bound $W(L,\ell)$ now follows by a compactness argument.  Were there no such $W$, then for every $W$ there would be a coloring 
$\chi _W$ of $[W]$ with no monochromatic progression of the required type.  For each $W$, some coloring of $[W]$ is the restriction of 
$\chi _{W'}$ for infinitely many $W' \geq W$, as there are finitely many colorings of $[W]$; these choices are consistent, so 
K\"onig's lemma provides a coloring $\chi$ of $\mathbb N$ each of whose restrictions to an $[W]$ is a restriction of some $\chi _{W'}$.  
Then $\chi$ has no monochromatic progression of the required type, since such a progression would lie in some $[W]$ and hence in $\chi _{W'}$.  
This contradicts the previous paragraph. 
\end{proof}

The important special case of Theorem \ref{thmCvdW} is when the system $(X,T)$  is \emph{minimal}, that is, when the only non-empty closed set $F\subset X$ with $TF\subset F$ is $F=X$. 
Two consequences of minimality will be used.  First, $T$ is onto: $T(X)$ is non-empty, closed as the continuous image of a compact set, 
and satisfies $T(T(X)) \subset T(X)$, so $T(X)=X$.  We do not assume that $T$ is invertible, and in the proof below we will simply 
choose some preimage.  Second, for every $x \in X$ and every non-empty open $U$ there is an $n \geq 0$ with $T ^{n}x \in U$, 
since the closure of $\{T ^{n}x \colon n \geq 0\}$ is a non-empty closed invariant set, hence all of $X$.

\begin{theorem}\label{thm:2}
    Let $(X,T)$ be a minimal topological dynamical system, $U\subset X$ a non-empty open set, $\ell \geq 1$, and $h \geq 0$. 
    Then there are $y\in X$ and $r\in b ^{h}\mathcal C\setminus \{0\}$ such that $U$ contains the arithmetic progression $y, T^{r}y, \ldots, T^{(\ell-1)r}y$.
\end{theorem}

\begin{proof}[Proof that Theorem \ref{thm:2} implies Theorem \ref{thmCvdW}]
This is a standard argument.   
Given a  topological dynamical system $(X,T)$, let $\mathbf F$ 
be the collection of non-empty closed $F\subset X$ which are $T$ invariant, thus $TF\subset F$.  
It is non-empty, as $X \in \mathbf F$, and it is ordered by inclusion.  
The intersection of a decreasing chain in $\mathbf F$ is again a non-empty closed invariant set, by compactness of $X$.  
Hence by Zorn's Lemma, applied to $\mathbf F$ ordered by reverse inclusion, $\mathbf F$ has a minimal element $F$, and $(F, T\vert _F)$ is a minimal system. 
Then, for any open cover $(U_\alpha)_{\alpha\in A}$ of $X$,  there is an $\alpha$ so that  $U_\alpha \cap F$ is non-empty, and it is relatively open in $F$. 
Hence 
$U_\alpha \cap F$ contains an  arithmetic progression of the desired type, by Theorem \ref{thm:2} applied to $(F,T\vert_F)$, so 
Theorem \ref{thmCvdW} follows. 
\end{proof}

\begin{proof}[Proof of Theorem \ref{thm:2}] We induct on the length $\ell$ of the arithmetic progression.  The case $\ell=1$ is trivial, as we may take any $y \in U$ and any $r \in b ^{h}\mathcal C \setminus \{0\}$.  So we take $\ell\geq 2$ and assume that the Theorem has been proved for $\ell-1$, for all minimal systems and all values of $h$.

Fix a minimal system $(X,T)$, a non-empty open $U$, and $h \geq 0$.  By minimality, finitely many of the sets $T ^{-n}U$ cover $X$; 
we fix such a finite open cover $(U_\alpha)_{\alpha \in A}$ of $X$, with each $U_\alpha = T ^{-n_\alpha}U$.  
It suffices to produce a progression $x, T ^{r}x, \ldots, T ^{(\ell-1)r}x$ inside one of the $U_\alpha$: applying $T ^{n_\alpha}$ to it 
produces a progression with the same common difference inside $U$.  

We now require the following Lemma, sometimes referred to as the `color focusing' Lemma. 

\begin{lemma}\label{lem:1}
   Let the notation and assumptions be as above. Then for any $J\geq 0$ there exist points $x_0,\cdots ,x_J$ in $X$, sets $U_{\alpha_0},\cdots ,U_{\alpha_J}$ in the open cover (not necessarily distinct), and integers $k_1,\cdots ,k_{J}$ such that
   \begin{enumerate}
       \item[(1)] $k_{j+1}+k_{j+2}+\cdots +k_{j'}\in b^{h}\mathcal C\setminus \{0\}$ for all $0\leq j< j'\leq J$;  
       \item[(2)]  $x_{j}\in U_{\alpha_j}$ for all $0\leq j \leq J$, and 
       \begin{align*}
T^{i(k_{j+1}+k_{j+2}+\cdots +k_{j'})}x_{j'}\in U_{\alpha_j}, \qquad 0\leq j<j'\leq J, \quad 1\leq i\leq \ell-1 . 
    \end{align*} 
   \end{enumerate}
\end{lemma}

\begin{proof}[Proof of Lemma \ref{lem:1}] We induct on $J$. For $J=0$, take any $x_0 \in X$ and any $U_{\alpha_0}$ containing it. Now suppose inductively that $J\geq 1$, and that we have already constructed points $x_0,\cdots ,x_{J-1}$, sets  $U_{\alpha_0},\cdots ,U_{\alpha_{J-1}}$, and integers $k_1,\cdots ,k_{J-1}$ with the requisite properties. 

We begin by fixing the neighborhood $V$ of $x_{J-1}$ in which the next progression will be found.  
Consider the maps 
\begin{equation} \label{e:finitelyMany}
T^{i(k_{j+1}+\cdots +k_{J-1})} , \qquad 0 \leq j <J, \quad 1 \leq i \leq \ell -1 . 
\end{equation}
The maps \eqref{e:finitelyMany} are finitely many and continuous, and each of them carries $x_{J-1}$ into the open set $U_{\alpha _j}$: 
for $j<J-1$ this is part (2) of the inductive hypothesis, applied to the pair $j<j'=J-1$, while for $j=J-1$ the exponent is the empty sum $0$, 
and the requirement is just that $x_{J-1}\in U_{\alpha _{J-1}}$, which is the first clause of (2).  
By continuity we may therefore choose a neighborhood $V$ of $x_{J-1}$ with 
\begin{equation} \label{e:Vsmall}
T^{i(k_{j+1}+\cdots +k_{J-1})}(V) \subset U_{\alpha _j}, \qquad 0 \leq j <J, \quad 1 \leq i \leq \ell -1 . 
\end{equation}
Note that \eqref{e:Vsmall} involves only data constructed at the previous stage, so $V$ is fixed before the next step.  

Now choose $h'\geq h$ so large that $k_1+\cdots +k_{J-1}<b^{h'}$. By Theorem \ref{thm:2} for progressions of length $\ell-1$, with the parameter $h'$, applied to the non-empty open set $V$, there are $y\in X$ and $r\in b^{h'}\mathcal C\setminus\{0\}$ with $y,T^{r}y,\cdots ,T^{(\ell-2)r}y \in V$.  
Set $k_J \coloneqq r$, choose $x_J \in X$ with $T ^{r}x_J = y$, which is possible as $T$ is onto, 
and let $U_{\alpha_J}$ be a member of the open cover that contains $x_J$.  

First we check (1).  Let $0 \leq j <J$.  The sum $k_{j+1}+\cdots +k_{J-1}$, which is the empty sum $0$ when $j=J-1$, 
lies in $b ^{h}\mathcal C$ by the inductive hypothesis and is less than $b^{h'}$; so its non-zero base $b$ digits lie in the positions $h, \ldots, h'-1$, and are in $D$.  
The digits of $k_J \in b^{h'}\mathcal C$ lie in the positions $h'$ and above, and are in $D$.  
The two summands are therefore supported on disjoint ranges of base $b$ digits, and adding them involves no carrying: 
$k_{j+1}+\cdots+k_J$ has all of its digits in $D$ and vanishing digits below the position $h$, that is, it lies in $b^{h}\mathcal C$.  It is non-zero, since $k_J$ is.  
The sums with $j' <J$ are unchanged from the previous stage.  

Now we check (2).  First, $x_J \in U_{\alpha_J}$ by the choice of $U_{\alpha _J}$.  Next, for $1 \leq i \leq \ell -1$,  
\begin{equation}
T^{i k_J}x_J=T^{(i-1)r}\bigl( T ^{r} x_J\bigr) = T ^{(i-1)r}y \in V ,
\end{equation}
since $y, T^{r}y, \ldots, T^{(\ell-2)r}y\in V$.  Consequently, for $0\leq j<J$ and $1\leq i\leq \ell-1$, 
\begin{align*}
    T^{i(k_{j+1}+\cdots +k_{J-1}+k_{J})}x_J = T^{i(k_{j+1}+\cdots +k_{J-1})}\bigl( T^{ik_J}x_J \bigr) 
    \in T^{i(k_{j+1}+\cdots +k_{J-1})}(V) \subset U_{\alpha _j} ,
\end{align*} 
the last inclusion being \eqref{e:Vsmall}.  The conditions in (2) with $j'<J$ are unchanged from the previous stage.  
This completes the induction on $J$, and so proves the lemma. 
\end{proof}

We then apply the above lemma with $J$ equal to the number of sets in the open cover. By the pigeonhole principle, we can then find some $0\leq j< j'\leq J$ such that $U_{\alpha_j}=U_{\alpha_{j'}}$. If we set $x:=x_{j'}$ and $r:=k_{j+1}+\cdots +k_{j'} \in b ^{h} \mathcal C \setminus \{0\}$, then  $x \in U_{\alpha_{j'}} = U_{\alpha _j}$ and $T ^{ir}x \in U_{\alpha _j}$ for $1 \leq i \leq \ell -1$.  That is, $U_{\alpha_j}$ contains the progression $x, T^{r}x , \ldots, T^{(\ell-1)r}x$. This completes the induction step for $\ell$, and so proves Theorem \ref{thm:2}. 
\end{proof}

\section{Compact Systems and Extensions} \label{sec:compact}

This section proves part (K) of Lemma \ref{l:Extensions}, in Theorem \ref{t:compactExtension}: a compact extension of a CSZ factor is CSZ.  
Taking the factor to be trivial, it shows that every compact system is CSZ.  
In \S\ref{sub:rootsOfUnity} we add Theorem \ref{t:compactSystem}, which is \emph{not} used in the proof of the main theorem.  
It treats the compact systems whose eigenvalues are all roots of unity, by a direct argument, and for them it gives more: 
a lower bound of $\int f ^{t+1}\;dm - \epsilon$, attained on a syndetic set of indices.  
Throughout, let $\mathcal C = \mathcal C_{b,D} = \{0=k_0<k_1<\cdots\}$ be a classical Cantor set, with $0\in D$, and $\delta = \lvert D\rvert$.

\subsection{Compact extensions}

\begin{theorem} \label{t:compactExtension}
Let $\mathcal D$ be a factor of $(X, \mathcal A, m, T)$ which is CSZ, and let $\mathcal E$ be a compact extension of $\mathcal D$.  
Then $\mathcal E$ is CSZ.  That is, for all integers $t \geq 1$, all $i \geq 0$, and all 
non-zero $0\leq f\leq 1$ in $L ^{\infty}(\mathcal E)$, we have 
\begin{equation} \label{e:compactExtensionConclusion}
\liminf_{N-M\to\infty}  \Expectation_{M\leq n<N} 
\int \prod _{s=0}^{t} T ^{s b ^{i} k_n} f\;dm >0. 
\end{equation}
\end{theorem}

The proof follows Furstenberg, Katznelson and Ornstein \cite{MR670131}, in the form presented by Moreira \cite{MoreiraNotes}*{\S 8.5}.  
We single out two features of that argument which make it the right one for our purposes.  
First, it works directly with the uniform averages that we require, whereas the van der Waerden argument of 
Furstenberg \cite{MR0498471} is naturally phrased with ordinary Ces\`aro averages.  
Second, and decisively, the return time that it produces is a \emph{translate} $n+a$ of the averaging parameter $n$, 
by an element $a$ of a fixed finite set, whereas the van der Waerden argument produces a \emph{dilate} $sn$, with $s$ 
ranging over a finite set.  The integers are closed under dilation, so for Szemer\'edi's theorem either argument serves.  
A Cantor set is not closed under dilation, but it is, in the following sense, closed under the translations we need: 
if $a \in \mathcal C\cap [0,b ^{i})$ and $n \in b ^{i}\mathcal C$, then by Lemma \ref{l:Diff}(ii), 
\begin{equation} \label{e:absorb}
n = b ^{i}k_u \quad\Longrightarrow\quad n+a = b ^{i}k_u + a = k _{\delta ^{i} u + \iota (a)} \in \mathcal C , 
\end{equation}
where $\iota (a) < \delta ^{i}$ denotes the index of $a$ in $\mathcal C$.  
That is, the translate stays in the Cantor set, and its index stays inside the block $[\delta ^{i}u , \delta ^{i}(u+1))$.  
This is why the argument below restricts the finite set of translations to the Cantor set, and the averaging parameter to a dilate of it.  

\begin{remark}
In particular, the proof below does not use the Cantor van der Waerden Theorem of \S\ref{sec:vdW}, 
which is retained as a result of independent interest.  
\end{remark}

We need this known Lemma. 

\begin{lemma} \label{l:localize} Let $\mathcal E\mid \mathcal D$ be a compact extension. 
Then for all $f\in L^2(\mathcal E\mid \mathcal D)$ and $0< \epsilon <1$, there is a set $A \in \mathcal D$ such that $m(A)>1- \epsilon $ and $f \mathbf{1}_{A}$ is conditionally compact. 
\end{lemma}

See \cite{MoreiraNotes}*{Lemma 8.4.3} for a proof.

\begin{proof}[Proof of Theorem \ref{t:compactExtension}]
We give the argument for the dilation $i=0$.  The case of a general $i$ is the same argument with $T$ replaced by $T ^{b ^{i}}$, 
and is set out in Remark \ref{r:dilation} at the end of the section.  
The hypothesis that $\mathcal D$ is CSZ enters exactly once, at \eqref{e:c0} below, and there for the same Cantor set $\mathcal C$ 
and a dilation fixed in the course of the proof.  

Fix an integer $t \geq 1$ and a non-zero $0 \leq f \leq 1$ in $L^2(\mathcal E \mid \mathcal D)$.   
We may and do assume that $f$ is conditionally compact.  Indeed, as $f \geq 0$ is not identically zero, we may choose $0<\eta <1$ so small 
that $m(A_\ast) > 1- \eta$ implies $\int_{A_\ast} f \;dm > 0$; let $A_\ast$ be as in Lemma \ref{l:localize} for this $\eta$, and replace $f$ by $f \mathbf 1_{A_\ast}$.  
This is legitimate because $0 \leq f \mathbf 1_{A_\ast} \leq f$ and all the factors in \eqref{e:compactExtensionConclusion} are non-negative, 
so a positive liminf for $f \mathbf 1_{A_\ast}$ gives one for $f$.  
Set 
\begin{equation}
g \coloneqq \mathbb E (f \mid \mathcal D), \qquad \delta _0 \coloneqq \int_X f \;dm >0 . 
\end{equation}
Now fix, once and for all, an $\epsilon >0$ small enough that 
\begin{equation} \label{e:epschoice}
c_1 \coloneqq ( \delta _0/2) ^{t+1} - t \epsilon >0 . 
\end{equation}
Thus $\epsilon$ depends only on $t$ and $\delta _0$; everything constructed below depends on $\epsilon$.  

\smallskip 
\noindent \textbf{Good pairs.}  
By conditional compactness there are $r \in \mathbb N$ and $f_1 , \ldots, f_r$ so that for almost every $x$ and every $n \in \mathbb N_0$, 
\begin{equation} \label{e:net}
\min_{1 \leq \ell \leq r} \lVert T^{n}f - f_\ell \rVert _{\mathcal E \mid \mathcal D}(x) < \epsilon /2 . 
\end{equation}
Call a pair $\{a, a'\}$ of distinct elements of $\mathcal C$ \emph{good at $x$} if 
\begin{equation}
\lVert T^{sa}f - T ^{sa'}f \rVert _{\mathcal E \mid \mathcal D}(x) < \epsilon , \qquad 0 \leq s \leq t , 
\end{equation}
and call a finite $F \subset \mathcal C$ \emph{unfortunate at $x$} if no pair in $F$ is good at $x$.  

\begin{lemma} \label{l:unfortunate}
Set $M \coloneqq r ^{t+1}+1$.  If $F \subset \mathcal C$ has $\lvert F\rvert \geq M$, then for almost every $x$, 
$F$ is not unfortunate at $x$.  
\end{lemma}

\begin{proof}
Let $Y$ be the set of full measure on which \eqref{e:net} holds for all $n$, and fix $x \in Y$.  
To each $a \in F$ associate the tuple $( \ell _0 , \ldots, \ell _t) \in [r] ^{t+1}$, where $\ell _s$ is least with 
$\lVert T^{sa}f - f _{\ell _s} \rVert _{\mathcal E\mid \mathcal D}(x)< \epsilon /2$.  
There are $r ^{t+1}$ such tuples, so if $\lvert F \rvert \geq M$ two elements $a \neq a'$ of $F$ share a tuple.  
The conditional triangle inequality then gives $\lVert T^{sa}f - T^{sa'}f\rVert _{\mathcal E\mid \mathcal D}(x)< \epsilon$ for $0 \leq s \leq t$. 
\end{proof}

\smallskip 
\noindent \textbf{A maximal unfortunate set.}  
For finite $F \subset \mathcal C$ put 
\begin{equation}
S(F) \coloneqq \{ x \colon F \textup{ is unfortunate at } x \} 
= \bigcap_{a \neq a' \in F} \ \bigcup_{0\leq s \leq t} 
\bigl\{ x \colon \lVert T^{sa}f-T^{sa'}f \rVert _{\mathcal E\mid \mathcal D}(x) \geq \epsilon \bigr\} .
\end{equation}
Conditional norms are $\mathcal D$ measurable, so $S(F) \in \mathcal D$.  Put 
\begin{equation}
Q(F) \coloneqq \bigl\{ x \in S(F) \colon x \notin S(G) \textup{ for every finite } G \subset \mathcal C \textup{ with } \lvert G\rvert > \lvert F\rvert \bigr\} . 
\end{equation}
As $\mathcal C$ is countable, it has only countably many finite subsets; so $Q(F)$ is a countable intersection of sets in $\mathcal D$, and is in $\mathcal D$.  
Singletons are unfortunate at every $x$, and by Lemma \ref{l:unfortunate} the sets $S(G)$ with $\lvert G\rvert \geq M$ are null.  
Hence almost every $x$ lies in $Q(F)$ for some finite $F \subset \mathcal C$ with $0< \lvert F\rvert < M$, namely for $F$ an unfortunate set at $x$ of maximal cardinality.  

Let $A \coloneqq \{ g > \delta _0/2\} \in \mathcal D$, and note that $\delta _0 = \int_A g + \int _{X \setminus A} g \leq m(A) + \delta _0/2$, 
so that $m(A) > \delta _0/2$.  
There are countably many finite subsets of $\mathcal C$, so we may fix one, $F$, with $m (Q(F) \cap A)>0$.  
For $x \in Q(F)$ no pair of $F$ is good at $x$, so for each pair $a \neq a'$ there is an $s$ with 
$\lVert T^{sa}f-T^{sa'}f\rVert _{\mathcal E\mid \mathcal D}(x) \geq \epsilon$; necessarily $s \in [t]$, since the case $s=0$ reads $0 \geq \epsilon$.  Let $F ^{(2)} \coloneqq \{ (a,a') \in F \times F \colon a \neq a'\}$, the ordered pairs of distinct elements.  
For each of the finitely many maps $q \colon F ^{(2)} \to [t]$, let $E_q$ be the set of $x \in Q(F) \cap A$ with 
\begin{equation} \label{e:qmap}
\lVert T^{q(a,a')a}f-T^{q(a,a')a'}f\rVert _{\mathcal E\mid \mathcal D}(x) \geq \epsilon , 
\qquad (a,a') \in F ^{(2)} . 
\end{equation}
Each $E_q$ is in $\mathcal D$, since conditional norms are $\mathcal D$-measurable, and by the previous sentence the sets $E_q$ cover $Q(F) \cap A$.  
So we may fix a map $q$ with $m(E_q)>0$, and we set $E \coloneqq E_q$.

\smallskip 
\noindent \textbf{Fixing the scale.}  
The set $F$ is a fixed finite subset of $\mathcal C$, so we may fix $i \geq 1$ with 
\begin{equation} \label{e:scale}
F \subset \mathcal C \cap [0, b ^{i}) . 
\end{equation}
This is the scale at which the argument will produce its return times, through \eqref{e:absorb}.  

\begin{lemma} \label{l:returnTime}
Let $u \geq 1$ and set $n = b ^{i}k_u$.  Let $x \in E \cap T ^{-n}E \cap \cdots \cap T ^{-tn}E$.  
Then there is an $a = a(u,x) \in F$ with 
\begin{equation} \label{e:good0}
\lVert T ^{s(n+a)}f - f \rVert _{\mathcal E\mid \mathcal D}(x) < \epsilon , \qquad 0 \leq s \leq t . 
\end{equation}
\end{lemma}

\begin{proof}
Let $a \neq a'$ in $F$ and set $s_0 = q(a,a') \in [t]$.  Since $T ^{s_0 n}x \in E$, \eqref{e:qmap} applies at the point $T ^{s_0n}x$.  
The conditional norm is equivariant, in that $\lVert T ^{v} h \rVert _{\mathcal E\mid \mathcal D}(x) = \lVert h \rVert _{\mathcal E\mid \mathcal D}(T ^{v}x)$, so that 
\begin{equation}
\lVert T ^{s_0(n+a)}f - T ^{s_0 (n+a')}f \rVert _{\mathcal E\mid \mathcal D}(x) 
= \lVert T ^{s_0a}f - T ^{s_0a'}f \rVert _{\mathcal E\mid \mathcal D}(T ^{s_0n}x) \geq \epsilon . 
\end{equation}
Thus no pair of distinct elements of $F + n$ is good at $x$.  

Now set $G \coloneqq (F+n) \cup \{0\}$.  By \eqref{e:scale} and \eqref{e:absorb}, $F + n \subset \mathcal C$, and $0 \in \mathcal C$, so $G \subset \mathcal C$.  
Moreover $n = b ^{i}k_u \geq b ^{i} > \max F$ as $u \geq 1$, so $0 \notin F+n$ and $\lvert G\rvert = \lvert F\rvert +1$.  
Since $x \in E \subset Q(F)$, the set $G$ is not unfortunate at $x$, so it contains a good pair.  By the previous paragraph that pair must involve $0$; 
say it is $\{n+a, 0\}$ with $a \in F$.  This is \eqref{e:good0}. 
\end{proof}

\smallskip 
\noindent \textbf{Conclusion.}  
Apply the CSZ property of $\mathcal D$ to the set $E$, with dilation $i$, to get  
\begin{equation} \label{e:c0}
c_0 \coloneqq \liminf_{N-M\to\infty}  \Expectation_{M\leq u<N} m (I_u)>0 , 
\qquad I_u \coloneqq E \cap T ^{-b ^{i}k_u}E \cap \cdots \cap T ^{-tb ^{i}k_u}E . 
\end{equation}
Let $B \coloneqq \{u \geq 1\colon m (I_u) > c_0/2\}$.  Splitting the average over $B$ and its complement, and using $m(I_u) \leq 1$, 
\begin{equation}
c_0 \leq \liminf_{N-M\to\infty} \frac{\lvert B \cap [M,N) \rvert}{N-M} + \frac {c_0}2 , 
\end{equation}
so that $B$ has lower uniform density at least $c_0/2$.  

Fix $u \in B$, write $n = b ^{i} k_u$, and let $x \in I_u$.  Lemma \ref{l:returnTime} provides $a = a(u,x) \in F$ satisfying \eqref{e:good0}.  
Write $d = n+a$.  Telescoping, $\prod _{s=0}^{t} T ^{sd}f - f ^{t+1}$ is a sum of $t$ terms, each of the form 
$P \cdot (T ^{sd}f - f) \cdot P'$ with $1\leq s \leq t$ and $P, P'$ products of the remaining factors, so bounded by $1$.  
Applying the conditional Cauchy--Schwarz inequality to each term, and then the conditional Jensen inequality, gives 
\begin{equation}
\mathbb E \Bigl[ \prod _{s=0}^{t} T ^{s(n+a)}f \;\Big\vert\; \mathcal D \Bigr](x) 
\geq \mathbb E \bigl( f ^{t+1} \mid \mathcal D\bigr)(x) - t \epsilon 
\geq g ^{t+1}(x) - t\epsilon 
\geq (\delta _0/2) ^{t+1} - t \epsilon = c_1 >0, 
\end{equation}
the last two steps because $I_u \subset E \subset A$, and by the choice \eqref{e:epschoice} of $\epsilon$.  
Fix an ordering of $F$, and let $a(u,x)$ be the least $a \in F$ for which \eqref{e:good0} holds at $x$.  
For each $a \in F$, the set of $x$ at which \eqref{e:good0} holds is in $\mathcal D$, as the conditional norms are $\mathcal D$-measurable; 
so $x \mapsto a(u,x)$ is $\mathcal D$-measurable on $I_u$, and it takes at most $\lvert F\rvert$ values.  
Hence there is an $a_u \in F$ for which $J_u \coloneqq \{ x \in I_u \colon a(u,x) = a_u\}$, a set in $\mathcal D$, has $m(J_u) \geq m(I_u)/\lvert F\rvert$.  
Since $f \geq 0$, 
\begin{equation} \label{e:lowerbound}
\int_X \prod _{s=0}^{t} T ^{s(n+a_u)}f \;dm 
\geq \int _{J_u} \mathbb E \Bigl[ \prod _{s=0}^{t} T ^{s(n+a_u)}f \;\Big\vert\; \mathcal D \Bigr] dm 
\geq \frac{m(I_u)\, c_1}{\lvert F\rvert} \geq \frac{c_0c_1}{2 \lvert F\rvert} \eqqcolon c_2 >0 . 
\end{equation}

Finally we convert \eqref{e:lowerbound} into a statement about the enumeration of $\mathcal C$.  
Since $0 \in D$, the first $\delta ^{i}$ elements of $\mathcal C$ are exactly $\mathcal C \cap [0,b ^{i})$, by Lemma \ref{l:Diff}(i).  
So $a_u = k _{\iota (a_u)}$ for a unique index $\iota (a_u) < \delta ^{i}$, and Lemma \ref{l:Diff}(ii) gives, as in \eqref{e:absorb}, 
\begin{equation}
n + a_u = b ^{i}k_u + k _{\iota (a_u)} = k _{v(u)} , \qquad v(u) \coloneqq \delta ^{i}u + \iota (a_u) . 
\end{equation}
Thus $v(u)$ lies in the block $[\delta ^{i}u , \delta ^{i}(u+1))$, and 
$\int_X \prod _{s=0}^{t} T ^{s k_{v(u)}}f \;dm \geq c_2$ for every $u \in B$, by \eqref{e:lowerbound}.  
Let $V \coloneqq \{ v(u) \colon u \in B\}$.  The blocks $[\delta ^{i}u, \delta ^{i}(u+1))$ are consecutive intervals of length $\delta ^{i}$ partitioning $\mathbb N_0$, 
so $u \mapsto v(u)$ is injective and $V$ meets the block of each $u \in B$ in exactly one point.  
If $\lceil P/\delta ^{i}\rceil \leq u < \lfloor Q/ \delta ^{i} \rfloor$ then the whole block of $u$ lies in $[P,Q)$, so 
\begin{equation} \label{e:count}
\bigl\lvert V \cap [P,Q) \bigr\rvert 
\geq \bigl\lvert B \cap [ \lceil P/\delta ^{i}\rceil , \lfloor Q/ \delta ^{i} \rfloor ) \bigr\rvert . 
\end{equation}
The interval on the right has length at least $(Q-P)/ \delta ^{i} -2$, which tends to infinity with $Q-P$.  
Dividing \eqref{e:count} by $Q-P$ and using that $B$ has lower uniform density at least $c_0/2$, 
\begin{equation}
\liminf_{Q-P\to\infty} \frac{\lvert V \cap [P,Q)\rvert}{Q-P} \geq \frac{1}{\delta ^{i}}\cdot \frac {c_0}2 . 
\end{equation}
The integrand $\int_X \prod _{s=0}^{t} T ^{s k_{v}}f \;dm$ is non-negative for every index $v$, and is at least $c_2$ for $v \in V$.  Therefore 
\begin{equation}
\liminf_{Q-P\to\infty} \Expectation_{P \leq v <Q} \int_X \prod _{s=0}^{t} T ^{s k_{v}}f \;dm 
\geq \frac{c_0c_2}{2 \delta ^{i}} >0 . 
\end{equation}
This is \eqref{e:compactExtensionConclusion} with $i=0$, and so completes the proof. 
\end{proof}

\begin{remark} \label{r:dilation}
To obtain \eqref{e:compactExtensionConclusion} for a general dilation $i_0 \geq 0$, run the proof above with $T$ replaced by 
$\widetilde T \coloneqq T ^{b ^{i_0}}$ throughout.  This is legitimate: $\mathcal E$ is a compact extension of $\mathcal D$ for $\widetilde T$, 
by Lemma \ref{l:powers}(1).  The conclusion reached is \eqref{e:compactExtensionConclusion} for $\widetilde T$ with dilation $0$, 
which is \eqref{e:compactExtensionConclusion} for $T$ with dilation $i_0$.  
The only appeal to the hypothesis on $\mathcal D$ is at \eqref{e:c0}, which now reads 
\begin{equation}
\liminf_{N-M\to\infty}  \Expectation_{M\leq u<N} 
m \bigl( E \cap T ^{-b ^{i_0+i}k_u}E \cap \cdots \cap T ^{-tb ^{i_0+i}k_u}E \bigr)>0 , 
\end{equation}
with $i$ the scale \eqref{e:scale}.  This is \eqref{e:CSZdef} for $\mathcal D$, for the same Cantor set $\mathcal C$, with dilation $i_0+i$, 
and so is part of the hypothesis.  

It is worth saying why the argument is organized in this way, rather than by asserting that $(X,\mathcal D,m,\widetilde T)$ is CSZ and quoting the case $i_0=0$.  
That assertion would be false in general.  The CSZ property quantifies over all Cantor sets, and the dilations it provides for a second Cantor set 
are powers of that set's base, not of $b$.  What is true, and all that is used, is the single instance displayed above.  
\end{remark}

\subsection{Systems whose eigenvalues are roots of unity} \label{sub:rootsOfUnity}

This subsection is independent of the proof of the main theorem.  

\begin{theorem} \label{t:compactSystem} 
Let $(X,\mathcal A, m,T)$ be a measure preserving system in which the functions $h$ satisfying $T ^{q}h = h$, for some integer $q \geq 1$, 
are dense in $L^2(X)$.  Then, for every $f\in L^\infty(X)$ with $0 \leq f \leq 1$, every $\epsilon >0$, and every integer $t \geq 0$, the set of indices 
$n$ such that  
\begin{equation} \label{e:compactSystem}
\int \prod _{s=0} ^{t} T ^{s k_n } f 
\;dm \geq \int f ^{t+1}\;dm - \epsilon 
\end{equation}
is syndetic.  
\end{theorem}

The hypothesis says exactly that $T$ has pure point spectrum and that all of its eigenvalues are roots of unity.  
Indeed, for a given $q$ the functions with $T ^{q}h=h$ form a $T$-invariant subspace on which $T ^{q}$ is the identity, 
and which is therefore spanned by eigenfunctions whose eigenvalues are $q$-th roots of unity; conversely, a finite linear combination of 
eigenfunctions whose eigenvalues are roots of unity is fixed by $T ^{q}$, for $q$ the least common multiple of their orders.  
Odometers are examples.  Such a system is compact, but a compact system need not satisfy the hypothesis: 
for an irrational rotation of the circle, $T ^{q}$ is ergodic for every $q \geq 1$, so the only functions fixed by some $T ^{q}$ are the constants.  

\begin{proof}
Let $\eta \coloneqq \epsilon /(2(t+1))$, and choose $h$ and $q \geq 1$ with $T ^{q}h = h$ and $\lVert f-h\rVert _2 < \eta$.  
Let $h' \coloneqq \min (1, \max (0, \operatorname{Re} h))$.  Then $T ^{q}h' = h'$, as $h'$ is a function of $h$; 
and $\lvert f - h'\rvert \leq \lvert f-h\rvert$ pointwise, because $z \mapsto \min(1,\max(0,\operatorname{Re} z))$ is $1$-Lipschitz and fixes the values of $f$, which lie in $[0,1]$.  
Hence $\lVert f-h'\rVert _1 \leq \lVert f - h'\rVert _2 < \eta$.  

For functions $a_0 , \ldots, a_t$ and $b_0 , \ldots, b_t$ with values in $[0,1]$, telescoping gives 
$\lvert \prod_s a_s - \prod_s b_s\rvert \leq \sum_s \lvert a_s - b_s\rvert$ pointwise.  As $T$ preserves $m$, it follows that, for every $n$, 
\begin{equation}
\Bigl\lvert \int \prod _{s=0} ^{t} T ^{s k_n } f \;dm - \int \prod _{s=0} ^{t} T ^{s k_n } h' \;dm \Bigr\rvert 
\leq (t+1) \lVert f - h'\rVert _1 < \frac \epsilon 2 , 
\end{equation}
and likewise $\int h' {}^{t+1}\;dm \geq \int f ^{t+1}\;dm - \epsilon/2$.  

Let $S_q$ be the set of indices $n$ with $k_n \equiv 0 \mod q$, which is syndetic by Lemma \ref{l:syndeticResidues}.  
For $n \in S_q$ we have $T ^{sk_n}h' = h'$ for every $s$, so $\int \prod_{s=0}^t T ^{s k_n } h' \;dm = \int h' {}^{t+1}\;dm$, and therefore 
\begin{equation}
\int \prod _{s=0} ^{t} T ^{s k_n } f \;dm > \int h' {}^{t+1}\;dm - \frac \epsilon 2 \geq \int f ^{t+1}\;dm - \epsilon . 
\end{equation}
That is, \eqref{e:compactSystem} holds on the syndetic set $S_q$.  
\end{proof}

%\bibliography{bibliography}

\end{document}